\documentclass[preprint,sort&compress,11pt]{elsarticle}
\usepackage[T1]{fontenc}
\usepackage{lmodern}
\usepackage{bbm}
\usepackage{bm}
\usepackage{amsmath}
\usepackage{amssymb}
\usepackage{amsfonts}
\usepackage{amsthm}
\usepackage{mathrsfs}

\usepackage{stmaryrd}
\SetSymbolFont{stmry}{bold}{U}{stmry}{m}{n}

\usepackage{pstricks}
\usepackage{pst-coil}
\usepackage{pst-3d}

\usepackage[colorlinks=true,linkcolor=blue,citecolor=blue,urlcolor=blue]{hyperref}

\makeatletter
\let\ps@pprintTitle\ps@empty
\makeatother

\allowdisplaybreaks[4]
\newcommand{\mm}{\mathrm}

\newcommand{\be}{\begin{equation}}
\newcommand{\bea}{\begin{equation}\begin{aligned}}
\newcommand{\beas}{\begin{equation*}\begin{aligned}}
\newcommand{\eeas}{\end{aligned}\end{equation*}}
\newcommand{\eea}{\end{aligned}\end{equation}}
\newcommand{\ee}{\end{equation}}

\begin{document}
\begin{frontmatter}
\title{On a thermodynamically consistent diffuse interface model for \\ incompressible two-phase flows with unmatched densities: \\
Energy equality and Lyapunov stability}
%%%%

\author[UR]{Harald Garcke}
\ead{harald.garcke@ur.de}

\author[FDU]{Maoyin Lv}
\ead{mylv22@m.fudan.edu.cn}

\author[FDU]{Hao Wu\corref{cor1}}
\ead{haowufd@fudan.edu.cn}

\cortext[cor1]{Corresponding author.}
%%%%%%%%%%%%%%%%%%%%%%%%%%%%%%%

\address[UR]{Fakult\"at f\"ur Mathematik, Universit\"at Regensburg, 93040 Regensburg, Germany}

\address[FDU]{School of Mathematical Sciences, Fudan University, 200433 Shanghai, China}

%%%%%%%%%%%%%%%%%
\begin{abstract}
\noindent We consider the initial-boundary value problem of a thermodynamically consistent diffuse interface model for incompressible two-phase flows with unmatched densities in a bounded domain $\Omega\subset\mathbb{R}^3$. Our first aim is to study the energy equality for global weak solutions by establishing mixed $L_t^qL_x^r$-regularity conditions on the velocity field, its gradient, and its time derivative, under which the global weak solution satisfies the energy equality for all times. The proof is based on the propagation of regularity for weak solutions to the convective Cahn--Hilliard equation with a physically relevant Flory--Huggins-type potential, combined with global mollification and boundary cut-off techniques. Next, we prove the existence and uniqueness of global strong solutions in the general setting with non-constant gradient energy coefficient and non-degenerate mobility, provided that the initial velocity is sufficiently small and the initial phase-field variable is a sufficiently small perturbation of a local minimizer of the free energy. This yields Lyapunov stability for each steady state consisting of a zero velocity together with a local energy minimizer. The proof relies on the energy equality for (local) strong solutions and the {\L}ojasiewicz--Simon approach.
\end{abstract}
\end{frontmatter}

%\linenumbers

\newtheorem{thm}{Theorem}[section]
\newtheorem{lem}{Lemma}[section]
\newtheorem{pro}{Proposition}[section]
\newtheorem{cor}{Corollary}[section]
\newproof{pf}{Proof}
\newdefinition{rem}{Remark}[section]
\newtheorem{definition}{Definition}[section]

\numberwithin{equation}{section}

%\tableofcontents

%%%%%%%%%%%%%%%%%%%%%%
\section{Introduction}

In this work, we consider the following diffuse interface model for a two-phase flow of two incompressible fluids with
different densities proposed by Abels, Garcke and Gr\"un in \cite{AGG2011} (henceforth referred to as the AGG model):
\begin{equation}\label{AGG}
\begin{cases}
\partial_t( \rho(\varphi) \mathbf{v})+ \text{div}(\mathbf{v}\otimes(\rho(\varphi)\mathbf{v}+{\mathbf{J}}))
-\mm{div}(2\nu(\varphi) \mathbb{D}\mathbf{v}
) +\nabla p &\\
\quad   =-
   \mm{div}( a(\varphi)\nabla \varphi\otimes\nabla \varphi  ) , &\text{in }\Omega\times(0,T), \\
\mathrm{div}\,\mathbf{v} =0, &\text{in }\Omega\times(0,T), \\
 \partial_t\varphi+{  \mathbf{v}}\cdot\nabla \varphi=\text{div}(b(\varphi)\nabla\mu), &\text{in }\Omega\times(0,T), \\
 \mu=a'(\varphi)\dfrac{|\nabla\varphi|^2}{2}-\text{div}(a(\varphi)\nabla\varphi)+\Psi'(\varphi), &\text{in }\Omega\times(0,T),
 \end{cases}
\end{equation}
subject to boundary and initial conditions:
\begin{align}
	\begin{cases}
		\mathbf{v}=\bm{0},\quad\partial_{\mathbf{n}}\varphi=\partial_{\mathbf{n}}\mu=0&\text{on }\partial\Omega\times(0,T),\\
		\mathbf{v}|_{t=0}=\mathbf{v}_0,\quad \varphi|_{t=0}=\varphi_0&\text{in }\Omega.
	\end{cases}\label{BC-IC}
\end{align}
Here, $\Omega\subset\mathbb{R}^3$ is a bounded domain with a sufficiently smooth boundary $\partial\Omega$,
and $T>0$ is a given final time.
The symbol $\mathbf{n}$ means the unit outward normal vector on $\partial\Omega$, and $\partial_\mathbf{n}$ denotes the outer normal derivative on the boundary.
In the coupled hydrodynamic system \eqref{AGG}, $\mathbf{v}:\Omega\times(0,T)\to \mathbb{R}^3$ denotes the volume averaged velocity of the fluid mixture and $p:\Omega\times(0,T)\to\mathbb{R}$ is the pressure. The symmetrized gradient of $\mathbf{v}$ is defined as $\mathbb{D}\mathbf{v}=(\nabla\mathbf{v}+(\nabla\mathbf{v})^{\top})/2$,
where the superscript $\top$ denotes the transposition. The phase-field variable $\varphi:\Omega\times (0,T)\to [-1,1]$ represents an order parameter given as the difference in the volume fractions of the fluid components such that the pure phases correspond to $\varphi=\pm 1$. In addition, $\mu:\Omega\times (0,T)\to \mathbb{R}$ denotes the chemical potential associated with $\varphi$. The averaged density $\rho$ and the averaged viscosity $\nu$ of the fluid mixture are defined via the typical linear form:
$$
\rho (\varphi) =\rho_1\frac{1+\varphi}{2}+\rho_2\frac{1-\varphi}{2},
\quad
\nu (\varphi) =\nu_1\frac{1+\varphi}{2}+\nu_2\frac{1-\varphi}{2},
$$
where the positive constants $\rho_1$, $\rho_2$ and $\nu_1$, $\nu_2$ are the homogeneous densities and viscosities of the two unmixed fluids, respectively. The flux term ${\mathbf{J}}$ given by
$$
 {\mathbf{J}}=-\frac{\rho_1-\rho_2}{2}b(\varphi)\nabla\mu,
$$
is related to diffusion caused by the density mismatch and is crucial for consistency with thermodynamics.
Here, the function $b:[-1,1]\to(0,\infty)$ denotes the (non-degenerate) Onsager mobility that measures the strength of diffusion. The total energy associated with the system \eqref{AGG} is given by
 \begin{align}
 	\mathcal{E}=E_{\text{kin}}(\mathbf{v},\varphi)+E_{\text{free}}(\varphi)
    =\int_\Omega\rho(\varphi)\frac{|\mathbf{v}|^2}{2}\,\mathrm{d}x
    +\int_\Omega \Big(a(\varphi)\frac{|\nabla\varphi|^2}{2}+\Psi(\varphi)\Big)\,\mathrm{d}x.\notag
    %\label{total-energy}
 \end{align}
In the expression of $E_{\text{free}}$, the positive function $a:[-1,1]\to (0,\infty)$ represents a gradient energy coefficient, while $\Psi:[-1,1]\to \mathbb{R}$ denotes the homogeneous free energy density, which usually has a double-well structure. A physically relevant example of $\Psi$ is the Flory--Huggins potential:
\begin{align}
    \Psi(s)=\Psi_0(s)-\frac{\Theta_0}{2}s^2
    =\frac{\Theta}{2}[(1+s)\text{ln}(1+s)+(1-s)\text{ln}(1-s)]-\frac{\Theta_0}{2}s^2,\quad s\in [-1,1],\label{FH}
\end{align}
where $0<\Theta<\Theta_0$ are constant parameters related to the system temperature and the critical temperature for phase separation. The logarithmic potential \eqref{FH} leads to a singular term in the equation for the chemical potential. We also note that the chemical potential $\mu$ is the variational derivative of the Ginzburg--Landau free energy $E_{\text{free}}(\varphi)$ with respect to $\varphi$, subject to the homogeneous Neumann boundary condition (see \eqref{BC-IC}).

The AGG model \eqref{AGG}--\eqref{BC-IC} provides a thermodynamically consistent diffuse interface framework to describe the evolution of two viscous, (macroscopically) immiscible, and incompressible Newtonian fluids with unmatched densities. This model adopts the volume-averaged velocity of the fluid mixture so that $\mathbf{v}$ is kept divergence-free. Moreover, it fulfills local and global dissipation inequalities and is frame indifferent (see \cite{AGG2011}). We recall that various types of diffuse interface models have been proposed for incompressible two-phase flows in the literature. The well-known ``Model H'' (see \cite{AMW,HH,Gur96}) can be recovered from the AGG model \eqref{AGG}--\eqref{BC-IC} by taking matched densities $\rho_1=\rho_2$ (see \cite{Gior21,Gior22} for a rigorous justification).
For the development of thermodynamically consistent diffuse interface models for multi-constituent flows with unmatched densities, we refer to \cite{AGG2011,Aki14,LT1998,Shen13,Sho18,ten24},
which are based on different choices of the mean velocity of the mixture and constitutive assumptions. Further discussions on the existing Navier--Stokes--Cahn--Hilliard models can be found in \cite{ten23}, where a unified framework was developed.

The AGG model \eqref{AGG}--\eqref{BC-IC} has been extensively studied from different aspects, see, for instance, \cite{AGG2018,ADGHP,ADG2013,AWJEE,Gior21,Gior22,AbelsGarckeGiorgini,GrasselliPoiatti}. Regarding progress in its nonlocal variant and other generalizations, we refer to  \cite{ADGJMFM,AGPCVPDE,ATM2AS,Frigeri16,Frigeri2021,GGGP,AGW2025} and the references therein. For sufficiently smooth solutions of the problem \eqref{AGG}–\eqref{BC-IC}, the following mass conservation and energy balance properties hold:
 \begin{align}
    & \int_\Omega \varphi(t)\,\mathrm{d}x=\int_\Omega \varphi_0\,\mathrm{d}x,\quad \forall\, t\geq 0,\label{mass-conv}\\
 	&\mathcal{E}(t)
    + 2 \int_0^t\int_\Omega \nu(\varphi(s))|\mathbb{D}\mathbf{v}(s)|^2\,\mathrm{d}x\,\mathrm{d}s
    +\int_0^t\int_\Omega b(\varphi(s))|\nabla\mu(s)|^2\,\mathrm{d}x\,\mathrm{d}s
    =\mathcal{E}(0),\quad \forall\, t\geq 0.
    \label{energy-equality}
 \end{align}
These properties play a crucial role in the analysis of solutions, including global existence, (weak-strong) uniqueness, and long-time behavior. The validity of \eqref{mass-conv} and \eqref{energy-equality} depends on the regularity properties of the solutions. In particular, due to the limited regularity, we only expect global weak solutions of problem \eqref{AGG}--\eqref{BC-IC} to satisfy an energy inequality (see \cite{ADG2013,ADGHP}). In \cite{AbelsGarckeGiorgini}, global regularity and long-time stabilization of global weak solutions were achieved in the special case with a constant mobility $b$ and a constant gradient energy coefficient $a$. There, after establishing the regularization of a global weak solution in finite time and the strict separation of the phase-field variable from the pure phases $\pm 1$, the authors demonstrated its convergence to a single equilibrium using the {\L}ojasiewicz--Simon inequality combined with the energy identity. In the recent work \cite{GrasselliPoiatti}, the authors introduced a different approach and proved the convergence to a single equilibrium for any global weak solution only satisfying an energy inequality, under the assumption that the mobility $b$ is a non-degenerate continuous function and the gradient energy coefficient $a$ is a positive constant.

The energy equality can rule out the possibility
of interior anomalous energy dissipation in the fluid mixture. This is well motivated from a physical perspective but remains an open problem, as the existence of sufficiently smooth global solutions for arbitrarily large initial data is unknown. Therefore, the aim of this study is two-fold:
\begin{itemize}
    \item[(i)] establish regularity criteria under which global weak solutions of the AGG model \eqref{AGG}--\eqref{BC-IC} satisfy the energy equality \eqref{energy-equality};
    \item[(ii)] explore the role of energy equality in the Lyapunov stability of energy-minimizing steady states (in the context of global strong solutions).
\end{itemize}

\subsection{Regularity criteria for the energy equality}

The lack of energy equality for the AGG model \eqref{AGG}--\eqref{BC-IC} is reminiscent of that for the Leray--Hopf solution to the Navier--Stokes equations for incompressible single-fluid flows. For the classical Navier--Stokes equations, various $L^q_tL^r_x$-type regularity criteria on the velocity field have been established to justify the energy conservation for global weak solutions, see, e.g., \cite{JL-Lions,Prodi,Lady,Serrin,Shinbrot}.
An alternative condition involving the velocity gradient instead of the velocity itself can be found in \cite{BY2019,BC2020}.
For recent progress, we refer to   \cite{Ber2023,Berselli2021,CheLuo,Leslie,CFS,Farwig} and the review article \cite{BY2025}.

Next, we recall some developments on the diffuse interface models for two-phase flows.
In \cite{LiangShuai2019}, the authors considered the compressible Cahn--Hilliard--Navier--Stokes system
in the periodic domain $\mathbb{T}^3$. They established a regularity condition for the energy conservation of global  weak solutions based on the commutator estimates and mollification approximations. For the incompressible Cahn--Hilliard--Navier--Stokes system (i.e., the ``Model H'') in $\mathbb{T}^3$, an $L^q_tL^r_x$-type regularity condition on the gradient of the velocity field was derived in \cite{LNS}, under an additional assumption on the continuity of $\Psi''$ on $[-1,1]$. Later in \cite{Geo}, the author relaxed the condition in \cite{LNS} and established the energy equality assuming that the velocity field is H\"older continuous in space.

In this study, we establish a regularity criterion
for the energy equality of the AGG model \eqref{AGG}--\eqref{BC-IC} with a constant gradient coefficient $a$ and a constant mobility $b$ in a smooth bounded domain $\Omega \subset \mathbb{R}^3$. Compared with the existing literature \cite{LNS,Geo}, the novelty of our work is threefold: (i) it addresses the general case of unmatched densities; (ii) it is applicable to general smooth bounded domains; and (iii) it accommodates the physically relevant Flory--Huggins potential \eqref{FH}.
The result is presented in Theorem \ref{criterion}. In the following, we briefly mention some features of the problem and key points of the proof.

(1) The regularity condition on $\mathbf{v}$ introduced in Theorem \ref{criterion} looks stronger than those in \cite{LNS,Geo}, where the specific densities of both fluid components are assumed to coincide. This is mainly due to the non-homogeneity of the fluid mixture.
When $\rho_1\neq \rho_2$, the momentum equation for the fluid velocity $\mathbf{v}$ contains an additional flux term $\mathbf{J}=-\rho'(\varphi)b(\varphi)\nabla\mu$, which is highly nonlinear. Consequently, $\mathbf{v}\otimes\mathbf{J}\in L^{\frac{8}{7}}(0,T;\mathbf{L}^{\frac{4}{3}}(\Omega))$, which yields further requirements on the regularity of test functions.
Moreover, the unmatched densities also lead to the time derivative of the weak solution, namely, $\partial_t(\mathbb{P}_\sigma(\rho(\varphi)\mathbf{v}))$,  belongs only to $L^{\frac{8}{7}}(0,T;\mathbf{W}^{-1,\frac{4}{3}}(\Omega))$, weaker than its counterpart in the ``Model H''. Indeed, the additional regularity conditions \eqref{conditional-regularity}, \eqref{conditional-regularity'} and the instantaneous regularity of $\varphi$ (see \cite[Theorem 1.3]{AbelsGarckeGiorgini} and Proposition \ref{weak-solution} below) imply that for any $\tau>0$, it holds
    \[ \partial_t(\rho(\varphi)\mathbf{v})\in L_{\mathrm{loc}}^{q'}([\tau/2,\infty);\mathbf{W}^{-1,r'}(\Omega)),
    \quad p\in L_{\mathrm{loc}}^{q'}([\tau/2,\infty);\mathbf{L}^{r'}(\Omega)),
    \]
where $q,r\geq2$, $1/q+1/q'=1$ and $1/r+1/r'=1$. The above properties enable us to test the momentum equation with test functions related to $\mathbf{v}$
   and to pass to the limit as the approximating parameters tend to zero.

(2) Compared to \cite{LNS,Geo}, we achieve our result in a general smooth bounded domain $\Omega \subset \mathbb{R}^3$ instead of the torus $\mathbb{T}^3$. The main idea is to apply techniques that involve a global mollifier and a boundary cut-off function to overcome the difficulties caused by the boundary (cf. \cite{CLWX2020}). When the fluid mixture is confined in a bounded domain, possible boundary effects can make the dissipative mechanism more complex, and the near-wall behavior of the solution may differ substantially from that in the bulk. This yields additional challenges to control the regularity of the solution near the boundary, which is necessary to establish the global energy balance. The basic strategy for handling the boundary effect is localization: introducing an additional cut-off function to separate the near-boundary region from the interior, see, e.g., \cite{Akramov,Bardos2019,BardosTiti,BardosTitiWiedemann,Drivas}. In recent work \cite{CLWX2020}, the authors introduced an approach that avoids assuming additional regularity of the velocity near the boundary and successfully applied it to both compressible and incompressible Navier–Stokes systems.

(3) In Theorem \ref{criterion}, we confine ourselves to the simplified case with a constant gradient coefficient $a$ and a constant mobility $b$. 
This technical restriction stems primarily from the currently available global regularity of the phase-field variable and the chemical potential $(\varphi,\mu)$. In \cite{LNS,Geo}, an essential structural assumption, i.e., $\Psi''\in C([-1,1])$, was imposed to obtain the energy equality. Unfortunately, the Flory--Huggins potential \eqref{FH}, which is of central interest, does not satisfy this continuity property for the second derivative of $\Psi$. To remove this assumption, we take advantage of the instantaneous propagation of regularity for $(\varphi,\mu)$, that is,
\begin{align*}
      &\varphi\in L^\infty(\tau,\infty;W^{2,6}(\Omega)) \cap H^1_{\mathrm{uloc}}([\tau,\infty);H^1(\Omega)),\\
     &\mu\in L^\infty(\tau,\infty;H^1(\Omega))\cap L^2_{\mathrm{uloc}}([\tau,\infty);H^3(\Omega)),
\end{align*}
    for any $\tau>0$, see \cite[Theorem 1.3]{AbelsGarckeGiorgini} and Proposition \ref{weak-solution} below. This allows us to overcome the difficulties caused by the singular Flory--Huggins potential \eqref{FH} and to pass to the limit as the approximating parameters tend to zero. However, the global regularity properties mentioned above for weak solutions $(\varphi,\mu)$ of the three-dimensional Cahn--Hilliard equation with divergence-free drift only hold when both the gradient coefficient $a$ and the mobility $b$ are positive constants. Even for the pure Cahn--Hilliard equation with a singular potential and non-constant coefficients, the global regularity for weak solutions in three dimensions remains open, cf. \cite{CGGG,CGGS}.
    Consequently, extending the regularity result in \cite[Theorem 1.3]{AbelsGarckeGiorgini} to the case of a non-constant gradient coefficient $a$ and mobility $b$, and thereby establishing the energy equality, remains an interesting and challenging problem.

%%%%%%%%%%%%%%%%%%%%%%%%%%%%%%%%%%%%%%%%%%%%%%%%%%%%%
\subsection{Global strong solutions and Lyapunov stability of steady states}

Since the AGG model \eqref{AGG}--\eqref{BC-IC} contains the Navier--Stokes equations driven by a capillary force term as a subsystem, it is natural to expect that the existence and uniqueness of a global strong solution can only be obtained under suitable smallness assumptions on the initial data, e.g., near a certain equilibrium. For the ``Model H'' with a singular potential $\Psi$ in three dimensions, the stability of local minimizers of the total energy was established in \cite{ZhouDX,Abels2009} via different methods both essentially based on the {\L}ojasiewicz--Simon approach \cite{LinLiu,LS83}, see also \cite{ZWH} for the case with a regular potential. In the recent work \cite{HeWu2024}, the authors analyzed a more complicated model for incompressible two-phase flows with matched densities, in which mechanisms of chemotaxis, active transport, and reaction rate are taken into account. They proved the existence and uniqueness of global strong solutions subject to suitably small initial data, which also implies the Lyapunov stability of an energy minimizing equilibrium. Concerning the AGG model \eqref{AGG}--\eqref{BC-IC}, the existence of strong solutions is more challenging due to difficulties from the unmatched densities, even with constant coefficients $a,b$. In \cite{Gior21}, the author established the existence and uniqueness of strong solutions in the two-dimensional case, locally in time for bounded domains and globally in time in the case of periodic boundary conditions. Subsequently, \cite{Gior22} extended the local well-posedness result to the three-dimensional case. The global well-posedness in any generic two-dimensional bounded domain was established in \cite{AbelsGarckeGiorgini}.

Building on recent progress \cite{CGGS} concerning the Cahn--Hilliard equation with a non-constant gradient energy coefficient $a$ and a non-degenerate mobility $b$, we establish the existence and uniqueness of global strong solutions to the hydrodynamic problem \eqref{AGG}--\eqref{BC-IC} in three dimensions in a general setting, provided that the initial velocity $\mathbf{v}_0$ is sufficiently small and the initial phase-field variable $\varphi_0$ is a small perturbation of a local minimizer $\varphi_\ast$ of the free energy $E_{\mm{free}}$. Hence, the Lyapunov stability of the corresponding equilibrium follows.
The result is summarized in Theorem \ref{global-strong}.

To prove global existence, a natural way is to derive uniform-in-time estimates for the local strong solution so that it can be extended to the whole interval $[0,\infty)$. Based on the energy equality satisfied by the strong solution that is sufficiently smooth, we achieve our goal by using the {\L}ojasiewicz--Simon approach and, in particular, the argument inspired by \cite{LinLiu} for incompressible nematic liquid crystal flows (see \cite{CGGS,GGW,HeWu2024,ZhouDX} for further developments on phase-field models). Due to the energy balance, we see that the lower-order norms of the solution $(\mathbf{v},\varphi)$ in $\mathbf{L}^2(\Omega)\times H^1(\Omega)$ remain uniformly bounded as time evolves. However, its higher-order norms in $\mathbf{H}^1(\Omega)\times H^3(\Omega)$ can grow and may even explode in finite time. Applying a new type of \L ojasiewicz--Simon inequality (see Lemma \ref{LS}) combined with the energy equality \eqref{energy-equality}, we show that a small change in the total energy $\mathcal{E}$ guarantees a small change in the lower-order terms $\|\mathbf{v}\|_{L^2}$ and $\|\varphi-\varphi_\ast\|_{(H^1)'}$. By interpolation, this ensures  that the solution $(\mathbf{v}, \varphi)$ stays in a suitable small neighborhood of $(\bm{0},\varphi_\ast)$ in $\mathbf{L}^2(\Omega)\times H^2(\Omega)$, so that the total energy $\mathcal{E}$ does not drop ``too much'' along the trajectory. By induction, we can eventually extend the local strong solution with a fixed time step to obtain a unique global strong solution that is uniformly bounded for all $t\geq0$.

\textit{Plan of the paper}. In Section \ref{main results}, we first introduce the notation and necessary assumptions and then present the main results. Section \ref{section 3} is devoted to the proof of Theorem \ref{criterion} on the regularity criterion for the energy equality. In Section \ref{section 4}, we prove Theorem \ref{global-strong} about the existence and uniqueness of global strong solutions and Lyapunov stability. Finally, in the Appendix \ref{Appendix}, we provide some technical tools that have been used in this study.

%%%%%%%%%%%%%%%%%%%%%%%%%%%%%%
\section{Main Results}
\label{main results}

\subsection{Preliminaries}

First, we introduce some necessary notation. Let $X$ be a (real) Banach space with the norm $\|\cdot\|_X$. Its dual space is denoted by $X'$, and the duality pairing is denoted by $\langle\cdot,\cdot\rangle_{X',X}$.
We use the bold letter $\mathbf{X}$ for a generic space of vectors or matrices, with each component belonging to $X$.
Given an interval $J\subseteq [0,\infty)$, $L^q(J;X)$ with $q\in[1,\infty]$ denotes the space consisting of Bochner measurable $q$-integrable/essentially bounded functions
with values in $X$.
For $q\in[1,\infty]$, $W^{1,q}(J;X)$ is the space of all $f\in L^q(J;X)$ with $\partial_t f\in L^q(J;X)$.
For $q=2$, we set $H^1(J;X)=W^{1,2}(J;X)$.
The set of continuous functions $f : J \to X$ is denoted by $C(J;X)$, and $BC(J; X)$ is the Banach space of all bounded and continuous $f : J \to X$ equipped with the supremum norm. Throughout the paper, we assume that $\Omega$ is a bounded domain in $\mathbb{R}^3$ with a sufficiently smooth boundary $\partial\Omega$.
For any $q\in [1,\infty]$, $L^q(\Omega)$ denotes the Lebesgue space with norm $\|\cdot\|_{L^q}$.
For $k\in\mathbb{Z}^+$, $q\in[1,\infty]$, $W^{k,q}(\Omega)$ denotes the Sobolev space with norm $\|\cdot\|_{W^{k,q}}$.
For $q=2$, we set $H^k(\Omega)=W^{k,2}(\Omega)$.
We also denote by $W^{k,q}_0(\Omega)$ the closure of $C_0^\infty(\Omega)$ in $W^{k,q}(\Omega)$ and set $W^{-k,q'}(\Omega)=(W^{k,q}_0(\Omega))'$ with $1/q+1/q'=1$.
Moreover, $f\in L^q_{\mathrm{loc}}([t,\infty);X)$ if
and only if $f\in L^q(t,s;X)$ for every $s>t$. The space $L^q_{\mathrm{uloc}}([t,\infty);X)$ consists of all measurable
functions $f:[t,\infty)\to X$ such that
\[\|f\|_{L^q_{\mathrm{uloc}}([t,\infty);X)}:=\sup_{s\geq t}\|f\|_{L^q(s,s+1;X)}<\infty.\]
We also define the generalized mean value over $\Omega$ of $f\in (H^1(\Omega))'$ as $ \overline{f}=|\Omega|^{-1}\langle f,1\rangle_{(H^1(\Omega))',H^1(\Omega)}$. If $f\in L^1(\Omega)$, then $\overline{f}=|\Omega|^{-1}\int_\Omega f\,\mathrm{d}x$.

Next, we introduce the Hilbert spaces of solenoidal vector-valued functions:
 \begin{align*}
 	&\mathbf{L}_\sigma^2(\Omega)=\big\{\mathbf{u}\in\mathbf{L}^2(\Omega)\,:\,
    \text{div}\,\mathbf{u}=0\text{ in a distributional sense in }\Omega,\ \mathbf{u}\cdot\mathbf{n}=0\text{ on }\partial\Omega\big\},\\
 	&\mathbf{H}_\sigma^1(\Omega)=\big\{\mathbf{u}\in\mathbf{H}^1(\Omega)\,:\,
    \text{div}\,\mathbf{u}=0\text{ in }\Omega,\ \mathbf{u}=\bm{0}\text{ on }\partial\Omega\big\}.
 \end{align*}
 We note that $\mathbf{L}_\sigma^2(\Omega)$ and $\mathbf{H}_\sigma^1(\Omega)$ correspond to the completion of $\mathbf{C}_{0,\sigma}^\infty(\Omega)$, namely the space of divergence-free vector fields in $\mathbf{C}_{0}^\infty(\Omega;\mathbb{R}^3)$, in the norm of $\mathbf{L}^2(\Omega)$ and $\mathbf{H}^1(\Omega)$, respectively.
 We use $(\cdot,\cdot)_{L^2}$ and $\|\cdot\|_{L^2}$ for the inner product and the norm in $\mathbf{L}^2_\sigma(\Omega)$, respectively. Thanks to Poincar\'{e}'s inequality, the space $\mathbf{H}_\sigma^1(\Omega)$ can be endowed with the inner product $(\mathbf{u},\mathbf{w})_{\mathbf{H}_\sigma^1}=(\nabla\mathbf{u},\nabla\mathbf{w})_{L^2}$ and the norm
 $\|\mathbf{u}\|_{\mathbf{H}_\sigma^1}=\|\nabla\mathbf{u}\|_{L^2}$.
 We recall the Korn inequality
 \begin{align}
 \|\nabla \mathbf{u}\|_{L^2}\leq \sqrt{2}\|\mathbb{D}\mathbf{u}\|_{L^2}\leq\sqrt{2}\|\nabla\mathbf{u}\|_{L^2} ,
 \quad\forall\, \mathbf{u}\in \mathbf{H}_\sigma^1(\Omega),\label{Korn}
 \end{align}
 which implies that $\|\mathbb{D}\mathbf{u}\|_{L^2}$ is a norm on $\mathbf{H}_\sigma^1(\Omega)$
 that is equivalent to $\|\mathbf{u}\|_{\mathbf{H}_\sigma^1}$.
 We introduce the space
 $$
 \mathbf{H}_\sigma^2(\Omega)=\mathbf{H}^2(\Omega)\cap \mathbf{H}_\sigma^1(\Omega)
 $$
 with inner product $(\mathbf{u},\mathbf{w})_{\mathbf{H}_\sigma^2}=(\mathbf{A}\mathbf{u},\mathbf{A}\mathbf{w})_{L^2}$
 and norm $\|\mathbf{u}\|_{\mathbf{H}_\sigma^2}=\|\mathbf{A}\mathbf{u}\|_{L^2}$,
 where $\mathbf{A}=\mathbb{P}_\sigma(-\Delta)$ is the Stokes operator
 and $\mathbb{P}_\sigma$ is the Leray projection from $\mathbf{L}^2(\Omega)$ onto $\mathbf{L}_\sigma^2(\Omega)$.
 There exists a constant $C>0$ such that
 $\|\mathbf{u}\|_{\mathbf{H}^2}\leq C\|\mathbf{u}\|_{\mathbf{H}_\sigma^2}$ for all $\mathbf{u}\in \mathbf{H}_\sigma^2(\Omega)$.

 Throughout this paper, we will use the following notation:
 \[\rho_\ast=\min\{\rho_1,\rho_2\},\quad\rho^\ast=\max\{\rho_1,\rho_2\},\quad\nu_\ast=\min\{\nu_1,\nu_2\},\quad\nu^\ast=\max\{\nu_1,\nu_2\}.\]
The symbols $C$, $C_i$ represent generic positive constants that can change from line to line. The specific dependence of these constants in terms of the data will be pointed out if necessary.

%%%%%%%%%%%%%%%%%%%%%%%%%%%%%%%%%%%%
\subsection{Main results}

We first report a result on the existence of global weak solutions to problem \eqref{AGG}--\eqref{BC-IC} on $[0,\infty)$ and the global regularity of weak solutions when the functions $a(\cdot)$ and $b(\cdot)$ are constants (cf. \cite{ADG2013,AbelsGarckeGiorgini}).
\begin{pro}\label{weak-solution}
Let $\Omega$ be a smooth bounded domain in $\mathbb{R}^3$. Assume that the gradient energy coefficient $a(\cdot)$ and the non-degenerate mobility $b(\cdot)$ are both positive constants (equal to one without loss of generality).
For any given initial data $\mathbf{v}_0\in \mathbf{L}_\sigma^2(\Omega)$ and $\varphi_0\in H^1(\Omega)$
with $\|\varphi_0\|_{L^\infty(\Omega)}\leq1$ and $|\overline{\varphi_0}|<1$, the AGG model \eqref{AGG}--\eqref{BC-IC} with the Flory--Huggins potential \eqref{FH} admits a global weak solution $(\mathbf{v},\varphi)$ on $\Omega\times[0,\infty)$ such that
\begin{itemize}
    \item [(1)] Regularity: The global weak solution $(\mathbf{v},\varphi)$ possesses the regularity properties
    \begin{align*}
		&\mathbf{v}\in L^\infty(0,\infty;\mathbf{L}_\sigma^2(\Omega))\cap L^2(0,\infty;\mathbf{H}_\sigma^1(\Omega)),\\
		&\varphi\in L^\infty(0,\infty;H^1(\Omega))\cap L_{\mathrm{uloc}}^2([0,\infty);W^{2,6}(\Omega))
        \cap L_{\mathrm{uloc}}^4([0,\infty);H^2(\Omega)),\\
        &\varphi\in L^\infty(\Omega\times(0,\infty))\ \
        \text{with} \ \ |\varphi(x,t)|<1\ \ \text{a.e. in }\Omega\times(0,\infty),\\
        & \mu\in L_{\mathrm{uloc}}^2([0,\infty);H^1(\Omega)),\quad
        \Psi_0'(\varphi)\in L_{\mathrm{uloc}}^2([0,\infty);L^6(\Omega)).
	\end{align*}

    \item [(2)] Weak formulations: The global weak solution $(\mathbf{v},\varphi)$ satisfies
    \begin{align}
		&-\int_0^{\infty}\int_\Omega \rho(\varphi)\mathbf{v}\cdot\partial_t \mathbf{w}\,\mathrm{d}x\,\mathrm{d}t
        +\int_0^{\infty}\int_\Omega\mathrm{div}(\rho(\varphi)\mathbf{v} \otimes\mathbf{v})\cdot\mathbf{w} \,\mathrm{d}x\,\mathrm{d}t\notag\\
        &\qquad+2\int_0^{\infty}\int_\Omega \nu(\varphi)\mathbb{D}\mathbf{v}:\mathbb{D}\mathbf{w}\,\mathrm{d}x \,\mathrm{d}t-\int_0^{\infty}\int_\Omega(\mathbf{v} \otimes{\mathbf{J}}):\nabla\mathbf{w}\,\mathrm{d}x\,\mathrm{d}t\notag\\
        &\quad=\int_0^{\infty}\int_\Omega\mu\nabla\varphi \cdot\mathbf{w}\,\mathrm{d}x\,\mathrm{d}t\label{0609-weak1}
	\end{align}
	for all $\mathbf{w}\in \mathbf{C}_0^\infty((0,\infty)\times\Omega)$ with $\mathrm{div}\,\mathbf{w}=0$,
	\begin{align}
		-\int_0^{\infty}\int_\Omega\varphi\partial_t\xi\,\mathrm{d}x\,\mathrm{d}t
        +\int_0^{\infty}\int_\Omega\mathbf{v} \cdot\nabla\varphi\xi\,\mathrm{d}x\,\mathrm{d}t
        =-\int_0^{\infty}\int_\Omega \nabla\mu\cdot\nabla\xi\,\mathrm{d}x \,\mathrm{d}t\notag
	\end{align}
	for all $\xi\in C_0^\infty((0,\infty);C^1(\overline{\Omega}))$, and
	\begin{align}
		&\mu=-\Delta\varphi+\Psi'(\varphi),\quad\text{a.e. in }\Omega\times(0,\infty).\notag
	\end{align}

    \item [(3)] Initial conditions: Almost everywhere in $\Omega$, it holds $(\mathbf{v},\varphi)|_{t=0}=(\mathbf{v}_0,\varphi_0)$.

    \item [(4)] Energy inequality: For all $t\in [s,\infty)$ and almost all $s\in [0,\infty)$ (including $s = 0$), it holds
    \begin{align}
		\mathcal{E}(t)+2\int_s^t\int_\Omega\nu(\varphi(\tau))|\mathbb{D}\mathbf{v}(\tau)|^2\,\mathrm{d}x\,\mathrm{d}\tau
        +\int_s^t\int_\Omega |\nabla\mu(\tau)|^2\,\mathrm{d}x\,\mathrm{d}\tau
        \leq \mathcal{E}(s).\label{energy-inequality}
	\end{align}

    \item [(5)] Instantaneous regularity of $(\varphi,\mu)$: For any $\tau>0$, we have
    \begin{align*}
        &\varphi\in L^\infty(\tau,\infty;W^{2,6}(\Omega)),
        \quad\partial_t\varphi\in L^2(\tau,\infty;H^1(\Omega)),\\
        &\mu\in L^\infty(\tau,\infty;H^1(\Omega))\cap L^2_{\mathrm{uloc}}([\tau,\infty);H^3(\Omega)), \\
        & \Psi_0'(\varphi)\in L^\infty(\tau,\infty;L^6(\Omega)).
    \end{align*}

\item [(6)] Eventual regularity of $\mathbf{v}$: There exists a large time $T_R$ such that
\begin{align}
    \mathbf{v}\in L^\infty(T_R,\infty;\mathbf{H}_\sigma^1(\Omega))\cap L^2(T_R,\infty;\mathbf{H}_\sigma^2(\Omega))\cap H^1(T_R,\infty;\mathbf{L}_\sigma^2(\Omega)).
     \notag
\end{align}
\end{itemize}	
\end{pro}

\begin{rem}
Proposition \ref{weak-solution}-$(5)$ implies that $(\varphi,\mu)$ becomes a global strong solution for $t>0$, i.e.,
    \begin{align}
        &\partial_t\varphi +\mathbf{v}\cdot\nabla\varphi =\Delta\mu\quad \text{and}\ \quad
        \mu=-\Delta\varphi+\Psi'(\varphi), \qquad\text{a.e. in }\Omega\times(0,\infty),\label{strong1}
    \end{align}
 with homogeneous Neumann boundary conditions $\partial_\mathbf{n}\varphi=\partial_\mathbf{n}\mu=0$ almost everywhere on $\partial \Omega \times (0,\infty)$. This will be helpful in the later analysis for the energy equality, see, e.g., Lemma \ref{epsilon-limit-lemma-1}.
\end{rem}

Our first result states a regularity criterion that ensures the energy equality \eqref{energy-equality}.
\begin{thm}
	\label{criterion}
	Let $(\mathbf{v},\varphi,\mu)$ be a global weak solution to problem \eqref{AGG}--\eqref{BC-IC} obtained in Proposition \ref{weak-solution}. Assume, in addition,
    \begin{align}
    \nabla\mathbf{v}\in L_{\mathrm{loc}}^q([0,\infty);\mathbf{L}^r(\Omega))\quad\text{and}\quad\mathbf{v}\in L_{\mathrm{loc}}^\frac{2q}{q-2}([0,\infty);\mathbf{L}^\frac{2r}{r-2}(\Omega)),
    \label{conditional-regularity}
    \end{align}
    as well as
    \begin{align}
        \partial_t\mathbf{v}\in L_{\mathrm{loc}}^{\frac{q}{q-1}}([0,\infty);\mathbf{L}^{{r}_\ast}(\Omega)),\quad \text{with}\ r_\ast:=
        \begin{cases}
             \ \dfrac{3r}{4r-3},&\text{for }r\in[2,3),\\[1mm]
             \ \widetilde{r}>1,&\text{for }r=3,\\
             \ 1,&\text{for }r>3,
        \end{cases}\label{conditional-regularity'}
    \end{align}
	where the exponents $q, r\geq 2$
	and when $q=2$ (resp. $r=2$), the expression $\frac{2q}{q-2}$ (resp. $\frac{2r}{r-2}$) should be understood as $\infty$. Then, the solution $(\mathbf{v}, \varphi, \mu)$ satisfies the following energy equality
\begin{align}
		\mathcal{E}(t)+2\int_0^t\int_\Omega\nu(\varphi(s))|\mathbb{D}\mathbf{v}(s)|^2\,\mathrm{d}x\,\mathrm{d}s
        +\int_0^t\int_\Omega |\nabla\mu(s)|^2\,\mathrm{d}x\,\mathrm{d}s
        = \mathcal{E}(0),\quad \forall\, t\geq 0.
        \label{energy}
	\end{align}	
\end{thm}

\begin{rem}\rm
Extensions of Theorem \ref{criterion} to the general case that involves a non-constant gradient energy coefficient $a$ and/or a non-degenerate mobility $b$ remain open.
\end{rem}

\begin{rem}\rm
In contrast to the classical Serrin-type criteria formulated
in terms of the velocity (see, e.g., \cite{JL-Lions,Prodi,Lady,Serrin,Shinbrot}), our criterion for the energy equality is primarily expressed
in terms of the velocity gradient. In the context of the Navier--Stokes equations, we refer to \cite{BC2020} for conditions involving velocity gradients and to \cite{BY2019,ZZJ} for further developments.
In our current setting, however, owing to the unmatched densities and boundary effects, we require stronger assumptions on the exponents $q$ and $r$ than those in \cite{BC2020}. Moreover, we additionally require the condition $\mathbf{v}\in L_{\mathrm{loc}}^\frac{2q}{q-2}([0,\infty);\mathbf{L}^\frac{2r}{r-2}(\Omega))$ in \eqref{conditional-regularity} with $q,r\geq 2$ and the condition \eqref{conditional-regularity'} on $\partial_t \mathbf{v}$.
It remains an open question whether the energy equality can be established under assumptions closer to those in \cite{BC2020}, or even under classical Serrin-type criteria for the velocity itself.
\end{rem}

\begin{rem}\rm
     In the case with unmatched densities, i.e., $\rho_1\neq \rho_2$, the assumption on $\mathbf{v}$ in \eqref{conditional-regularity} together with the fact $\nabla \mu\in L^2(0,\infty;\mathbf{L}^2(\Omega))$ indicates that
     $$
     \mathbf{v}\otimes\mathbf{J}\in L^{q'}_{\mathrm{loc}}([0,\infty);\mathbf{L}^{r'}(\Omega))
     \quad \text{with}\quad
     \frac{1}{q}+\frac{1}{q'}=1,\quad \frac{1}{r}+\frac{1}{r'}=1.
     $$
     It also implies that $\mathbf{v}\in L^{2q'}_{\mathrm{loc}}([0,\infty);\mathbf{L}^{2r'}(\Omega))$, and thus $\rho(\varphi)\mathbf{v}\otimes\mathbf{v}\in L^{q'}_{\mathrm{loc}}([0,\infty);\mathbf{L}^{r'}(\Omega))$ due to the boundedness of $\rho(\varphi)$. These observations, together with the assumption \eqref{conditional-regularity'} and Proposition \ref{weak-solution}-$(5)$, enable us to reconstruct the pressure with the regularity property $p\in L^{q'}_{\mathrm{loc}}([\tau/2,\infty);L^{r'}(\Omega))$ for all $\tau>0$ and establish the refined weak formulation \eqref{refined-weak-form}; see Section \ref{app-for} for more details.
They also help us handle the influence of boundaries, especially when taking the limit as $\delta\to0$, where $\delta>0$ is a parameter related to the thickness of the cut-off at the boundary (see \eqref{delta-limit-2}).
On the other hand, by exploring interpolation inequalities, we find that \eqref{conditional-regularity} can be replaced by
$$
\nabla \mathbf{v}\in L^4_{\mathrm{loc}}([0,\infty);\mathbf{L}^\frac{12}{5}(\Omega))\quad \text{or}\quad \nabla \mathbf{v}\in L_{\mathrm{loc}}^\frac{10}{3}([0,\infty);\mathbf{L}^3(\Omega)),
$$
both of which involve conditions only on $\nabla \mathbf{v}$. This is because the aforementioned condition on $\nabla \mathbf{v}$ together with the property $\mathbf{v}\in L^\infty(0,\infty;\mathbf{L}_\sigma^2(\Omega))$ is sufficient to control the condition on $\mathbf{v}$ in \eqref{conditional-regularity}.
\end{rem}

\begin{rem}
Returning to the case with matched densities, i.e., $\rho_1=\rho_2$, the condition involving the velocity gradient $\nabla\mathbf{v}$ in \eqref{conditional-regularity} alone is sufficient to ensure the energy equality, if
\begin{align}
		\dfrac{5}{q}+\dfrac{6}{r}=5,\quad \text{if }  2\leq  r\leq \dfrac{12}{5}; \quad \text{or } \ \ \ q\geq 2,\quad \text{if }r=3.
        \label{index}
\end{align}
Indeed, when $\rho_1=\rho_2$, the additional flux term $\mathbf{v}\otimes\mathbf{J}$ in the momentum equation vanishes. Consequently, the condition \eqref{index} together with the weak formulation \eqref{0609-weak1} yields
$$
\partial_t\mathbf{v}=\partial_t\mathbb{P}_\sigma\mathbf{v}\in L_{\mathrm{loc}}^{q'}([0,\infty);\mathbf{W}^{-1,r'}(\Omega)).
$$
This regularity property is enough to reconstruct the pressure $p$ and establish a suitable weak formulation.
The condition \eqref{index} corresponds to subcases of the conditions proposed in \cite{LNS} for the classical ``Model H'' on the torus $\mathbb{T}^3$. The requirement $q,r\geq 2$ is kept to handle the pressure term and the boundary effects of a general bounded domain $\Omega$. We also note that in \cite{LNS}, two additional assumptions such as $\Psi''\in C([-1,1])$ and $\mathbf{v}\in C([0,\infty);\mathbf{L}_\sigma^2(\Omega))$ were imposed to obtain the energy equality, which are not necessary in our argument.
\end{rem}

To study the Lyapunov stability of steady states, we shall work within the framework of global strong solutions. This allows us to handle a more general scenario with non-constant coefficients $a,b$. In this context, we impose the following hypotheses:
\begin{itemize}
    \item [$\mathbf{(H)}$] The gradient energy coefficient $a:[-1,1]\to(0,\infty)$ satisfies
    \[a\in C^{2}([-1,1]):\ 0<a_\ast\leq a(s)\leq a^\ast,\quad\forall \,s\in [-1,1],\]
    and the non-degenerate mobility $b:[-1,1]\to (0,\infty)$ satisfies
    \[b\in C^{2}([-1,1]):\ 0<b_\ast\leq b(s)\leq b^\ast,\quad\forall \,s\in [-1,1].\]
\end{itemize}
\noindent Let us consider the following nonlocal elliptic boundary value problem:
\begin{align}
	\begin{cases}
		-\text{div}(a(\psi)\nabla\psi)+\dfrac{a'(\psi)}{2}|\nabla\psi|^2+\Psi'(\psi)
        =\overline{\dfrac{a'(\psi)}{2}|\nabla\psi|^2+\Psi'(\psi)},&\text{in }\Omega,\\
		\partial_{\mathbf{n}}\psi=0,&\text{on }\partial\Omega,
	\end{cases}\label{station}
\end{align}
subject to the constraint $\overline{\psi}=\overline{\varphi_0}$.
For any $m\in(-1,1)$, we denote the set of steady states as
\[
\mathcal{S}_m=\big\{\psi\in H^2(\Omega)\ \text{with}\ \Psi'(\psi)\in L^2(\Omega): \psi\text{ solves }\eqref{station} \text{ and }\overline{\psi}=m\big\}.
\]
Next, given the set
\[
\mathcal{V}_m=\big\{\psi\in H^{1}(\Omega):\|\psi\|_{L^\infty}\leq 1,\, \overline{\psi}=m\big\},
\]
we say that a function $\psi _\ast\in \mathcal{V}_m$ is a local minimizer of the free energy $E_{\mathrm{free}}$ in $\mathcal{V}_m$,
if there exists a positive constant $\lambda_\ast$ such that
for any $\psi\in \mathcal{V}_m$ satisfying $\|\psi-\psi _\ast\|_{H^1}\leq \lambda_\ast$, it holds
	\begin{align}
		E_{\mathrm{free}}(\psi _\ast)\leq E_{\mathrm{free}}(\psi).\notag%\label{local}
	\end{align}
According to \cite[Lemma 10.1]{CGGS}, we have
$$\psi_\ast\in\mathcal{S}_m.$$

We now state the second result of this study on the existence and uniqueness of global strong solutions and Lyapunov stability of energy minimizing steady states.
\begin{thm}
	\label{global-strong}
Let $\Omega\subset\mathbb{R}^3$ be a smooth bounded domain, and let the assumption $\mathbf{(H)}$ be satisfied.
Suppose that $a(\cdot)$ is real analytic on the open interval $(-1,1)$, $m\in(-1,1)$
and $\varphi _\ast\in \mathcal{V}_m$ is a local minimizer of $E_{\mathrm{free}}$. Moreover, assume that the initial data $(\mathbf{v}_0,\varphi _0)\in \mathbf{H}_\sigma^1(\Omega)\times (\mathcal{V}_m\cap H^2(\Omega))$, $\partial_\mathbf{n}\varphi_0=0$ on $\partial\Omega$ and
    \begin{align}
      \|\mathbf{v}_0\|_{H^1}\leq M,
      \quad
      \Big\|-\mathrm{div}(a(\varphi_0)\nabla\varphi_0)+a'(\varphi_0)\frac{|\nabla\varphi_0|^2}{2}+\Psi'(\varphi_0)\Big\|_{ H^1}\leq M,
      \label{initial-mu0}
	\end{align}
for some large constant $M>0$.
Then, for any $\epsilon>0$, there exist constants $\eta_1$, $\eta_2\in(0,1)$, depending on $\epsilon$ and $M$, such that if
	\begin{align*}
		&\|\mathbf{v}_0\|_{L^2}\leq \eta_1,\quad\|\varphi _0-\varphi _\ast\|_{H^2}\leq \eta_2,
	\end{align*}
then problem \eqref{AGG}--\eqref{BC-IC} admits a unique global strong solution $(\mathbf{v},\varphi,\mu,p)$ satisfying
	\begin{align*}
		&\mathbf{v}\in BC([0,\infty);\mathbf{H}^1_\sigma(\Omega))
        \cap L_{\mathrm{uloc}}^2([0,\infty);\mathbf{H}^2_\sigma(\Omega))
        \cap H_{\mathrm{uloc}}^1([0,\infty);\mathbf{L}^2_\sigma(\Omega)),\\
		&\varphi\in L^\infty(0,\infty;H^{3}(\Omega)), \\
        & \partial_t\varphi\in L^\infty(0,\infty;(H^{1}(\Omega))')\cap L_{\mathrm{uloc}}^2([0,\infty);H^1(\Omega)),\\
		&\mu\in L^\infty(0,\infty;H^1(\Omega))\cap L_{\mathrm{uloc}}^2([0,\infty);H^3(\Omega)),
        \\
        &p\in L_{\mathrm{uloc}}^2([0,\infty);H^1(\Omega)),
	\end{align*}
and $\partial_\mathbf{n}\varphi=\partial_\mathbf{n}\mu=0$ almost everywhere on $\partial\Omega\times (0,\infty)$, $(\mathbf{v},\varphi)|_{t=0}=(\mathbf{v}_0,\varphi_0)$ almost everywhere in $\Omega$.
Furthermore, it holds
	\[\|\varphi(t)-\varphi _\ast\|_{H^2}\leq \epsilon,\quad\forall\,t\geq0,\]
    which entails the Lyapunov stability of the local minimizer $\varphi_*$.
\end{thm}
 Finally, we have the following characterization on the long-time behavior of global strong solutions:
 \begin{cor}
 	\label{conv-equ}
    Let $(\mathbf{v},\varphi,\mu)$ be the global strong solution to problem \eqref{AGG}--\eqref{BC-IC} constructed in Theorem \ref{global-strong}. Then, there exists an equilibrium  $\varphi_\infty\in H^3(\Omega)\cap \mathcal{S}_m$, uniquely determined as the asymptotic limit of the trajectory, such that
 	\[
    \lim_{t\to \infty}\Big(\|\mathbf{v}(t)\|_{H^1}+\|\varphi(t)-\varphi _\infty\|_{H^2}\Big)=0.
    \]
 	In addition, for some $\theta\in (0,1/2)$ depending on $\varphi _\infty$, it holds
 	\[\|\mathbf{v}(t)\|_{L^2} +\|\varphi(t)-\varphi _\infty\|_{H^1}\leq C (1+t)^{-\frac{\theta}{1-2\theta}},\quad\forall\, t\geq 0,\]
    where $C>0$ depends on $\|\varphi _0\|_{H^2}$, $\|\varphi _\ast\|_{H^2}$, $m$, $M$ and $\Omega$.
\end{cor}
\begin{rem}\rm
Based on the regularity and uniform-in-time boundedness of the global strong solution, Corollary \ref{conv-equ} can be proved by using the {\L}ojasiewicz--Simon approach and the energy equality \eqref{energy-equality}, with minor modifications to the arguments in \cite{AbelsGarckeGiorgini,CGGS,GGW,HeWu2024}. Hence, the details are left to the interested readers.
In general, the asymptotic limit $\varphi_\infty$ and the local minimizer $\varphi_*$ do not coincide. If, in addition, $\varphi_*$ is an isolated minimizer, then it is (locally) asymptotically stable, and hence $\varphi_\infty=\varphi_*$.
\end{rem}
\begin{rem}\rm
Since the special form of the Flory--Huggins potential \eqref{FH} is not used in our analysis, the conclusions of Theorems \ref{criterion}, \ref{global-strong} and Corollary \ref{conv-equ} can be further extended to singular potentials of the following general form (cf. \cite{CGGG,CGGS,GGW,HeWu2024}):
\begin{equation}
\Psi(r)=\Psi_{0}(r)-\frac{\Theta_{0}}{2}r^2,\nonumber
\end{equation}
where $\Psi_{0}\in C\big([-1,1]\big)\cap C^{3}(-1,1)$ satisfies
$$\lim_{r\to \pm 1} \Psi_{0}'(r)=\pm \infty\quad \text{and}\quad \Psi_{0}''(r)\ge \Theta\ \ \text{for all}\ r\in (-1,1),$$
with $\Theta_0\in \mathbb{R}$ and $\Theta>0$. For Theorem \ref{global-strong} and Corollary \ref{conv-equ}, we should also assume that $\Psi_0$ is analytic in $(-1,1)$ in order to apply the {\L}ojasiewicz--Simon inequality (see Lemma \ref{LS}). This assumption of analyticity is automatically fulfilled by \eqref{FH}.
\end{rem}

%%%%%%%%%%%%%%%%%%%%%%%%%%%%%%%%%%%%%%%%%%%%%%%%%
\section{Regularity Criterion on Energy Equality}
\label{section 3}
In this section, we prove Theorem \ref{criterion} on the regularity criterion of the energy equality \eqref{energy}.
Roughly speaking, the proof is divided into three steps:
\begin{itemize}
\item[(1)] reconstruct the
pressure and derive suitable globally mollified identities, see Section \ref{app-for};
\item[(2)] pass to the limit in the mollification
parameter, see Section \ref{sec:3.2};
\item[(3)] pass to the limit in the spatial and temporal cutoff parameters and finally recover the energy equality, see Section \ref{sec:3.3}.
\end{itemize}

%%%%%%%%%%%%%%%%%%%%%%%%%%%%%%%%%%
\subsection{Approximated formulas}
\label{app-for}

The key step is to reconstruct the pressure $p$
and rewrite the weak formula \eqref{0609-weak1} in a suitable form. Under the additional regularity \eqref{conditional-regularity} for $\mathbf{v}$,
following an argument similar to those in \cite{ADG2013,LNS}, we have
\begin{equation}
    \begin{aligned}
        &\rho(\varphi)\mathbf{v}\otimes \mathbf{v}\in L^{q'}_{\mathrm{loc}}([0,\infty);\mathbf{L}^{r'}(\Omega)),
        \\
        &\mathbb{D}\mathbf{v}\in L_{\mathrm{loc}}^{q'}([0,\infty);\mathbf{L}^{r'}(\Omega)),\\
        &\mathbf{v}\otimes\nabla\mu\in L_{\mathrm{loc}}^{q'}([0,\infty);\mathbf{L}^{r'}(\Omega)),
        \\
        &\mu\nabla\varphi\in L_{\mathrm{loc}}^2([0,\infty);\mathbf{L}^{\frac{3}{2}}(\Omega))\subset
        L_{\mathrm{loc}}^{q'}([0,\infty);\mathbf{W}^{-1,r'}(\Omega)),
    \end{aligned}
    \label{pressure-regularity}
\end{equation}
where
$$\frac{1}{q}+\frac{1}{q'}=1\quad \text{and}\quad \frac{1}{r}+\frac{1}{r'}=1\quad \text{with}\ \ q,r\geq 2.$$
Using the assumptions \eqref{conditional-regularity}, \eqref{conditional-regularity'}, Proposition \ref{weak-solution}-$(5)$ and the Sobolev embedding in three dimensions
\begin{align*}
    \mathbf{W}^{1,r}(\Omega)\hookrightarrow \mathbf{L}^{r_\ast'}(\Omega)=
    \begin{cases}
        \mathbf{L}^{\frac{3r}{3-r}}(\Omega),&\text{for }r\in[2,3),\\
        \mathbf{L}^{\frac{\widetilde{r}}{\widetilde{r}-1}}(\Omega)\text{ for any }\widetilde{r}>1,&\text{for }r=3,\\
        \mathbf{L}^\infty(\Omega),&\text{for }r>3,
    \end{cases}
\end{align*}
 where the exponent $r_\ast$ is defined as in \eqref{conditional-regularity'} and $1/r_\ast+1/r_\ast'=1$, we can check that for any $\tau\in (0,1)$, $T\in (\tau,\infty)$,
\begin{align}
        \partial_t(\rho(\varphi)\mathbf{v})&=\rho'(\varphi)\partial_t\varphi\mathbf{v}+\rho(\varphi)\partial_t\mathbf{v} \notag \\
        &\in L^{q'}(\tau/2,T;\mathbf{L}^{r'}(\Omega))+L^{q'}(\tau/2,T;\mathbf{L}^{r_\ast}(\Omega)) \notag \\
        &\subset L^{q'}(\tau/2,T;\mathbf{W}^{-1,r'}(\Omega)).
        \label{v_t}
\end{align}
Indeed, by the boundedness of $\rho'(\varphi)$, $\rho(\varphi)$, the assumptions \eqref{conditional-regularity}, \eqref{conditional-regularity'} and Proposition \ref{weak-solution}-$(5)$ on $\partial_t \varphi$, it holds
 \begin{align*}
     & \|\rho'(\varphi)\partial_t\varphi\mathbf{v}\|_{L^{q'}(\tau/2,T;\mathbf{L}^{r'}(\Omega))}\\
     &\quad \leq \|\rho'(\varphi)\|_{L^\infty(\Omega\times(0,T))}\|\partial_t \varphi\|_{L^2(\tau/2,T;L^2(\Omega))}\|\mathbf{v}\|_{L^{\frac{2q}{q-2}}(\tau/2,T;\mathbf{L}^{\frac{2r}{r-2}}(\Omega))}<\infty,
 \end{align*}
 and
 \begin{align*}
     \|\rho(\varphi)\partial_t\mathbf{v}\|_{L^{q'}(\tau/2,T;\mathbf{L}^{r_\ast}(\Omega))}
     &\leq \|\rho(\varphi)\|_{L^\infty(\Omega\times(0,T))}\|\partial_t\mathbf{v}\|_{L^{q'}(\tau/2,T;\mathbf{L}^{r_\ast}(\Omega))}<\infty.
 \end{align*}
 For any fixed $\tau\in (0,1)$, $T\in (\tau,\infty)$, it is convenient to introduce the functional $\mathcal{G}$ given by
 \begin{align*}
     \langle\mathcal{G},\mathbf{w}\rangle:&=-\int_0^T\int_\Omega \rho(\varphi)\mathbf{v}\cdot\partial_t\mathbf{w}\,\mathrm{d}x\,\mathrm{d}t+\int_0^T\int_\Omega \text{div}(\rho(\varphi)\mathbf{v}\otimes\mathbf{v})\cdot\mathbf{w}\,\mathrm{d}x\,\mathrm{d}t\\
     &\quad+2\int_0^T\int_\Omega \nu(\varphi)\mathbb{D}\mathbf{v}:\mathbb{D}\mathbf{w}\,\mathrm{d}x\,\mathrm{d}t-\int_0^T\int_\Omega (\mathbf{v}\otimes\mathbf{J}):\nabla\mathbf{w}\,\mathrm{d}x\,\mathrm{d}t\\
     &\quad-\int_0^T\int_\Omega \mu\nabla\varphi\cdot\mathbf{w} \,\mathrm{d}x\,\mathrm{d}t,\qquad\forall\,\mathbf{w}\in \mathbf{C}_0^\infty((\tau/2,T)\times\Omega).
 \end{align*}
 Then, \eqref{pressure-regularity} and \eqref{v_t} implies that
 the functional $\mathcal{G}$ belongs to $L^{q'}(\tau/2,T;\mathbf{W}^{-1,r'}(\Omega))$. Since
 \[ \langle\mathcal{G},\mathbf{w}\rangle=0,\quad\forall\,\mathbf{w}\in C_0^\infty((\tau/2,T);\mathbf{C}_{0,\sigma}^\infty(\Omega)),\]
 according to \cite[Chapter IV, Lemma 1.4.2]{Sohr} (see also \cite[Theorem 5.4]{Abels2019}), there exists a unique function
 $$p\in L^{q'}(\tau/2,T;L^{r'}(\Omega))$$
 such that
\[ \langle\mathcal{G}(t),\mathbf{w}\rangle=\int_\Omega
p(t)\,\text{div}\,\mathbf{w}\,\mathrm{d}x,\quad\forall\,\mathbf{w}\in \mathbf{C}_0^\infty((\tau/2,T)\times\Omega),\]
and
\[\int_\Omega p(t)\,\mathrm{d}x=0,\quad\text{for a.a. }t\in(\tau/2,T).\]
Hence, for any $T\in (2\tau,\infty)$, we can conclude
\begin{align}
    &-\int_0^T\int_\Omega\rho(\varphi) \mathbf{v}\cdot\partial_t\mathbf{w}\,\mathrm{d}x\,\mathrm{d}t
    +\int_0^T\int_\Omega\text{div}(\rho(\varphi)\mathbf{v}\otimes\mathbf{v})\cdot\mathbf{w}\,\mathrm{d}x\,\mathrm{d}t\notag\\
    &\qquad+\int_0^T\int_\Omega 2\nu(\varphi)\mathbb{D}\mathbf{v}: \mathbb{D}\mathbf{w}\,\mathrm{d}x\,\mathrm{d}t-\int_0^T\int_\Omega(\mathbf{v}\otimes{\mathbf{J}}): \nabla\mathbf{w}\,\mathrm{d}x\,\mathrm{d}t\notag\\
    &\quad=\int_0^T\int_\Omega\mu\nabla\varphi \cdot\mathbf{w}\,\mathrm{d}x\,\mathrm{d}t
     +\int_0^T\int_\Omega p\,\text{div}\,\mathbf{w}\,\mathrm{d}x\,\mathrm{d}t,\quad
     \forall\, \mathbf{w}\in \mathbf{C}_0^\infty((\tau/2,T)\times\Omega).
     \label{refined-weak-form}
\end{align}
For $0<\varepsilon<\tau/2$, according to \eqref{refined-weak-form} and applying the same argument as in \cite[Section 4]{CLWX2020}, we obtain
\begin{align}
        &\partial_t\Big(\xi_0 (\rho(\varphi)\mathbf{v})^{\varepsilon}
        +\sum_{i=1}^k\xi_i(\rho(\widetilde{\varphi})\widetilde{\mathbf{v}})_i^{\varepsilon}\Big)
        +\Big(\xi_0 \mathrm{div}(\rho(\varphi)\mathbf{v}\otimes\mathbf{v})^{\varepsilon}
        +\sum_{i=1}^k\xi_i\mathrm{div}(\rho(\widetilde{\varphi})\widetilde{\mathbf{v}}\otimes\widetilde{\mathbf{v}})^{\varepsilon}_i\Big)\notag\\
        &\qquad +\Big(\xi_0\mathrm{div}(\mathbf{v}\otimes\mathbf{J})^{\varepsilon}
        +\sum_{i=1}^k \xi_i\mathrm{div}(\widetilde{\mathbf{v}}\otimes\widetilde{\mathbf{J}})_i^{\varepsilon}\Big)
        \notag \\
        &\qquad
        -\Big(\xi_0 \mathrm{div}(2\nu(\varphi)\mathbb{D}\mathbf{v})^{\varepsilon}
        +\sum_{i=1}^k\xi_i \mathrm{div}(2\nu(\widetilde{\varphi})\mathbb{D}\widetilde{\mathbf{v}})_i^{\varepsilon}\Big)
        \notag\\
        &\qquad-\Big(\xi_0(\mu\nabla\varphi)^{\varepsilon}
        +\sum_{i=1}^k\xi_i (\widetilde{\mu}\nabla\widetilde{\varphi})_i^{\varepsilon}\Big)
        +\Big(\xi_0\nabla {p}^{\varepsilon}+\sum_{i=1}^k \xi_i \nabla \widetilde{p}^{\varepsilon}_i\Big)
        =\boldsymbol{0},\label{mollify-eq1}
\end{align}
for almost all $(x,t)\in\Omega\times(\tau,T-\tau)$, where the definition of approximations due to mollifiers $f^\varepsilon$, $\widetilde{f}_i^\varepsilon$ and the partition of unity $\{\xi_i\}$ can be found in Appendix \ref{APP-modif}.

Like in \cite{CLWX2020,LNS}, for given small constants $\tau,\delta>0$, we define the cut-off functions  $\chi_\tau\in W^{1,\infty}(0,T)$ such that
\begin{align}
	\chi_\tau(t)=
	\begin{cases}
		\displaystyle\frac{t-\tau}{\tau},&\tau\leq t\leq 2\tau,\\
		1,&2\tau\leq t\leq T-2\tau,\\[1mm]
		\displaystyle\frac{T-t-\tau}{\tau},& T-2\tau\leq t\leq T-\tau,\\
		0,&\text{otherwise},
	\end{cases}\notag
\end{align}
and $\phi_\delta(x)\in C_0^1(\Omega)$ satisfying
\begin{align}
    \begin{cases}
        0\leq \phi_\delta(x)\leq 1,\ \ \phi_\delta(x)=1
        \text{ if }x\in\Omega \text{ and } \text{dist}(x,\partial\Omega)\geq\delta,\\
        \phi_\delta\to1 \ \text{as}\ \delta\to0,\ \
        \text{and} \ \ |\nabla\phi_\delta|\leq \displaystyle\frac{2}{\text{dist}(x,\partial\Omega)}.
    \end{cases}\notag
\end{align}
Hence, $\chi_\tau \phi_\delta[\mathbf{v}]^{\varepsilon}$ is a legitimate test function,
where $[\mathbf{v}]^{\varepsilon}$ is defined as in \eqref{mollify-profile} with $f=\mathbf{v}$.
Multiplying \eqref{mollify-eq1} by $\chi_\tau \phi_\delta[\mathbf{v}]^{\varepsilon}$
and integrating it over $\Omega\times(0,T)$, we get
 \begin{align}
        &\int_0^T\int_\Omega\chi_\tau\phi_\delta[\mathbf{v}]^{\varepsilon}\cdot\partial_t\Big(\xi_0 (\rho(\varphi)\mathbf{v})^{\varepsilon}
        +\sum_{i=1}^k\xi_i(\rho(\widetilde{\varphi})\widetilde{\mathbf{v}})_i^{\varepsilon}\Big)\,\mathrm{d}x\,\mathrm{d}t\notag\\
        &\quad+\int_0^T\int_\Omega\chi_\tau\phi_\delta[\mathbf{v}]^{\varepsilon}\cdot\Big(\xi_0 \mathrm{div}(\rho(\varphi)\mathbf{v}\otimes\mathbf{v})^{\varepsilon}
        +\sum_{i=1}^k\xi_i\mathrm{div}(\rho(\widetilde{\varphi})\widetilde{\mathbf{v}}\otimes\widetilde{\mathbf{v}})^{\varepsilon}_i\Big)\,\mathrm{d}x\,\mathrm{d}t\notag\\
        &\quad+\int_0^T\int_\Omega\chi_\tau\phi_\delta[\mathbf{v}]^{\varepsilon}\cdot\Big(\xi_0\mathrm{div}(\mathbf{v}\otimes\mathbf{J})^{\varepsilon}
        +\sum_{i=1}^k \xi_i\mathrm{div}(\widetilde{\mathbf{v}}\otimes\widetilde{\mathbf{J}})_i^{\varepsilon}\Big)\,\mathrm{d}x\,\mathrm{d}t\notag\\
        &\quad-\int_0^T\int_\Omega\chi_\tau\phi_\delta[\mathbf{v}]^{\varepsilon}\cdot\Big(\xi_0 \mathrm{div}(2\nu(\varphi)\mathbb{D}\mathbf{v})^{\varepsilon}
        +\sum_{i=1}^k\xi_i \mathrm{div}(2\nu(\widetilde{\varphi})\mathbb{D}\widetilde{\mathbf{v}})_i^{\varepsilon}\Big)\,\mathrm{d}x\,\mathrm{d}t\notag\\
        &\quad-\int_0^T\int_\Omega\chi_\tau\phi_\delta[\mathbf{v}]^{\varepsilon}\cdot\Big(\xi_0(\mu\nabla\varphi)^{\varepsilon}
        +\sum_{i=1}^k\xi_i (\widetilde{\mu}\nabla\widetilde{\varphi})_i^{\varepsilon}\Big)\,\mathrm{d}x\,\mathrm{d}t\notag\\
        &\quad+\int_0^T\int_\Omega\chi_\tau\phi_\delta[\mathbf{v}]^{\varepsilon}\cdot\Big(\xi_0\nabla p^{\varepsilon}
        +\sum_{i=1}^k \xi_i \nabla \widetilde{p}^{\varepsilon}_i\Big)\,\mathrm{d}x\,\mathrm{d}t=0.
        \label{mollify-eq2}
    \end{align}
    Next, recalling \eqref{strong1},
    we test the equation $\partial_t\varphi+\mathbf{v}\cdot\nabla \varphi = \Delta \mu$ by $\chi_\tau\phi_\delta\mu$, then from the point-wise relation $\mu=-\Delta \varphi +\Psi'(\varphi)$ almost everywhere
    in $\Omega\times (0,\infty)$, we can deduce that
    \begin{align}
        0&=\int_0^T\int_\Omega\chi_\tau \phi_\delta\mu \partial_t\varphi\,\mathrm{d}x\,\mathrm{d}t
        +\int_0^T\int_\Omega \chi_\tau\phi_\delta\mu(\mathbf{v}\cdot\nabla\varphi)\,\mathrm{d}x\,\mathrm{d}t
        \\
        &\quad -\int_0^T\int_\Omega \chi_\tau\phi_\delta\mu\Delta\mu\,\mathrm{d}x\,\mathrm{d}t\notag\\
        &=\int_0^T\int_\Omega \chi_\tau\phi_\delta(-\Delta\varphi+\Psi'(\varphi))\partial_t\varphi\,\mathrm{d}x\,\mathrm{d}t
        +\int_0^T\int_\Omega \chi_\tau\phi_\delta\mu(\mathbf{v}\cdot\nabla\varphi)\,\mathrm{d}x\,\mathrm{d}t\notag\\
        &\quad-\int_0^T\int_\Omega \chi_\tau\phi_\delta\mu\Delta\mu\,\mathrm{d}x\,\mathrm{d}t.\label{mollify-eq}
    \end{align}

\subsection{\texorpdfstring{Passage to the limit as $\varepsilon\to0$}{Passage to the limit as \varepsilon \to 0}}
\label{sec:3.2}
In the following, we treat the terms in \eqref{mollify-eq2} one by one and pass to the limit as $\varepsilon\to 0$.
By interpolation, the condition \eqref{conditional-regularity} together with the fact $\mathbf{v}\in L^\infty(0,\infty;\mathbf{L}^2_\sigma(\Omega))$ implies that
$$
\mathbf{v}\in L^4_{\mathrm{loc}} ([0,\infty);\mathbf{L}^4(\Omega)).
$$
This property will be useful in the subsequent analysis.

\begin{lem}
\label{epsilon-limit-lemma-1}
    Let the assumptions in Theorem \ref{criterion} be satisfied. Then
    for fixed $\tau$ and $\delta$, the first three terms on the left-hand side of  \eqref{mollify-eq2} satisfy
    \begin{align}
        &\lim_{\varepsilon\to0}\int_0^T\int_\Omega\chi_\tau\phi_\delta[\mathbf{v}]^{\varepsilon}\cdot\partial_t\Big(\xi_0 (\rho(\varphi)\mathbf{v})^{\varepsilon}
        +\sum_{i=1}^k\xi_i(\rho(\widetilde{\varphi})\widetilde{\mathbf{v}})_i^{\varepsilon}\Big)\,\mathrm{d}x\,\mathrm{d}t\notag\\
        &\qquad+\lim_{\varepsilon\to0}\int_0^T\int_\Omega\chi_\tau\phi_\delta[\mathbf{v}]^{\varepsilon}\cdot\Big(\xi_0 \mathrm{div}(\rho(\varphi)\mathbf{v}\otimes\mathbf{v})^{\varepsilon}
        +\sum_{i=1}^k\xi_i\mathrm{div}(\rho(\widetilde{\varphi})\widetilde{\mathbf{v}}\otimes\widetilde{\mathbf{v}})^{\varepsilon}_i\Big)\,\mathrm{d}x\,\mathrm{d}t\notag\\
        &\qquad+\lim_{\varepsilon\to0}\int_0^T\int_\Omega\chi_\tau\phi_\delta[\mathbf{v}]^{\varepsilon}\cdot\Big(\xi_0\mathrm{div}(\mathbf{v}\otimes\mathbf{J})^{\varepsilon}
        +\sum_{i=1}^k \xi_i\mathrm{div}(\widetilde{\mathbf{v}}\otimes\widetilde{\mathbf{J}})_i^{\varepsilon}\Big)\,\mathrm{d}x\,\mathrm{d}t\notag\\
        &\quad=-\frac{1}{2}\int_0^T\int_\Omega\chi_\tau'\phi_\delta \rho(\varphi)|\mathbf{v}|^2\,\mathrm{d}x\,\mathrm{d}t
        -\frac{1}{2}\int_0^T\int_\Omega\chi_\tau\rho(\varphi)(\mathbf{v}\cdot\nabla\phi_\delta)|\mathbf{v}|^2\,\mathrm{d}x\,\mathrm{d}t\notag\\
        &\qquad+\frac{1}{2}\int_0^T\int_\Omega\chi_\tau\rho'(\varphi)(\nabla\mu\cdot\nabla\phi_\delta)|\mathbf{v}|^2\,\mathrm{d}x\,\mathrm{d}t.\label{epsilon-limit-1}
    \end{align}
\end{lem}

\begin{proof}
    We consider the first limit on the left-hand side of \eqref{epsilon-limit-1}. A direct calculation yields that
    \begin{align}
        &\int_0^T\int_\Omega\chi_\tau \phi_\delta[\mathbf{v}]^{\varepsilon}\cdot\partial_t\Big(\xi_0 (\rho(\varphi)\mathbf{v})^{\varepsilon}
        +\sum_{i=1}^k\xi_i (\rho(\widetilde{\varphi}) \widetilde{\mathbf{v}})_i^{\varepsilon}\Big)\,\mathrm{d}x\,\mathrm{d}t\notag\\
        &\quad=\int_0^T\int_\Omega \chi_\tau\phi_\delta[\mathbf{v}]^{\varepsilon}\cdot\partial_t(\rho(\varphi)[\mathbf{v}]^\varepsilon)\,\mathrm{d}x\,\mathrm{d}t
        \notag \\
        &\qquad +\int_0^T\int_\Omega \chi_\tau\phi_\delta[\mathbf{v}]^{\varepsilon}\cdot\partial_t([\rho(\varphi)\mathbf{v}]^\varepsilon-\rho(\varphi)[\mathbf{v}]^\varepsilon)\,\mathrm{d}x\,\mathrm{d}t\notag\\
        &\quad= \left(\int_0^T\int_\Omega \chi_\tau\phi_\delta[\mathbf{v}]^{\varepsilon}\cdot\partial_t(\rho(\varphi)[\mathbf{v}]^\varepsilon)\,\mathrm{d}x\,\mathrm{d}t\right.
        \notag \\
        &\qquad \left.+\int_0^T\int_\Omega \chi_\tau\phi_\delta\xi_0[\mathbf{v}]^{\varepsilon}\cdot\partial_t((\rho(\varphi)\mathbf{v})^\varepsilon-\rho(\varphi)\mathbf{v}^\varepsilon)\,\mathrm{d}x\,\mathrm{d}t\right) \notag\\
        &\qquad+\sum_{i=1}^k \int_0^T\int_\Omega\chi_\tau\phi_\delta\xi_i[\mathbf{v}]^{\varepsilon}
        \cdot\partial_t ((\rho(\widetilde{\varphi})\widetilde{\mathbf{v}})_i^\varepsilon
        -\rho({\varphi}) \widetilde{\mathbf{v}}_i^\varepsilon)\,\mathrm{d}x\,\mathrm{d}t\notag\\
        &\quad =: J_{\varepsilon,1} +\sum_{i=0}^k I_{\varepsilon,1}^{(i)}.
        \notag
    \end{align}
    For the term $J_{\varepsilon,1}$, according to the properties of regularity
    \begin{align}
         \mathbf{v}\in L^4(0,T;\mathbf{L}^4(\Omega)),\quad\partial_t\varphi\in L^2(\tau/2,T;H^1(\Omega))\label{property-0318-1}
    \end{align}
    and the property of the mollifier, it holds
\begin{align*}
    J_{\varepsilon,1}
    &=\int_0^T\int_\Omega\chi_\tau \phi_\delta\rho'(\varphi)\partial_t\varphi|[\mathbf{v}]^{\varepsilon}|^2\,\mathrm{d}x\,\mathrm{d}t
    +\int_0^T\int_\Omega\chi_\tau \phi_\delta\rho(\varphi)\partial_t\Big(\frac{1}{2}|[\mathbf{v}]^{\varepsilon}|^2\Big)\,\mathrm{d}x\,\mathrm{d}t\\
    &=\int_0^T\int_\Omega\chi_\tau \phi_\delta\rho'(\varphi)\partial_t\varphi|[\mathbf{v}]^{\varepsilon}|^2\,\mathrm{d}x\,\mathrm{d}t
    -\frac{1}{2}\int_0^T\int_\Omega\partial_t(\chi_\tau\phi_\delta\rho(\varphi))|[\mathbf{v}]^{\varepsilon}|^2\,\mathrm{d}x\,\mathrm{d}t\\
    &=\frac{1}{2}\int_0^T\int_\Omega\chi_\tau\phi_\delta\rho'(\varphi)\partial_t\varphi|[\mathbf{v}]^{\varepsilon}|^2\,\mathrm{d}x\,\mathrm{d}t
    -\frac{1}{2}\int_0^T\int_\Omega\chi_\tau'\phi_\delta\rho(\varphi)|[\mathbf{v}]^{\varepsilon}|^2\,\mathrm{d}x\,\mathrm{d}t\\
    &\to \frac{1}{2}\int_0^T\int_\Omega\chi_\tau\phi_\delta\rho'(\varphi)\partial_t\varphi|\mathbf{v}|^2\,\mathrm{d}x\,\mathrm{d}t
    -\frac{1}{2}\int_0^T\int_\Omega\chi_\tau'\phi_\delta\rho(\varphi)|\mathbf{v}|^2\,\mathrm{d}x\,\mathrm{d}t
    \quad\text{as }\varepsilon\to0.
\end{align*}
Concerning the term $I_{\varepsilon,1}^{(i)}$, from the commutator estimates \eqref{commutator-convergence-1}, \eqref{commutator-convergence-2},
and \eqref{property-0318-1},
we see that as $\varepsilon\to0$,
\begin{align*}
    |I_{\varepsilon,1}^{(0)}|
    &\leq C\|[\mathbf{v}]^\varepsilon\|_{{L}^4(\tau,T-\tau;\mathbf{L}^4(\Omega))}
    \|\partial_t((\rho(\varphi) \mathbf{v})^\varepsilon
    -\rho(\varphi)\mathbf{v}^\varepsilon)\|_{L^\frac{4}{3}(\tau,T-\tau;\mathbf{L}^\frac{4}{3}(V_0))}\to0,\\
    |I_{\varepsilon,1}^{(i)}|
    &\leq C\|[\mathbf{v}]^\varepsilon\|_{{L}^4(\tau,T-\tau;\mathbf{L}^4(\Omega))}
    \|\partial_t((\rho(\widetilde{\varphi}) \widetilde{\mathbf{v}})_i^\varepsilon
    -\rho({\varphi})\widetilde{\mathbf{v}}_i^\varepsilon)\|_{L^\frac{4}{3}(\tau,T-\tau;\mathbf{L}^\frac{4}{3}(V_i))}\\
    & \to0,\quad\text{for }i=1,...,k,
\end{align*}
Therefore, we can conclude
\begin{align}
    \lim_{\varepsilon\to0}\Big( J_{\varepsilon,1}+\sum_{i=0}^k I_{\varepsilon,1}^{(i)}\Big)
    & =\frac{1}{2}\int_0^T\int_\Omega\chi_\tau\phi_\delta\rho'(\varphi)\partial_t\varphi|\mathbf{v}|^2\,\mathrm{d}x\,\mathrm{d}t
    \notag \\
    &\quad -\frac{1}{2}\int_0^T\int_\Omega\chi_\tau'\phi_\delta\rho(\varphi)|\mathbf{v}|^2\,\mathrm{d}x\,\mathrm{d}t.
    \notag
\end{align}

Next, for the convective terms, we have
    \begin{align}
         &\int_0^T\int_\Omega\chi_\tau \phi_\delta[\mathbf{v}]^{\varepsilon}
         \cdot\Big(\xi_0 \mathrm{div}(\rho(\varphi)\mathbf{v}\otimes\mathbf{v})^{\varepsilon}
         +\sum_{i=1}^k\xi_i\mathrm{div}(\rho(\widetilde{\varphi})\widetilde{\mathbf{v}}\otimes\widetilde{\mathbf{v}})^{\varepsilon}_i\Big)
         \,\mathrm{d}x\,\mathrm{d}t\notag\\
         &\quad=\int_0^T\int_\Omega \chi_\tau\phi_\delta[\mathbf{v}]^{\varepsilon}
         \cdot\Big(\xi_0 \mathrm{div}(\rho(\varphi)\mathbf{v}\otimes\mathbf{v})^{\varepsilon}
         +\sum_{i=1}^k\xi_i\mathrm{div}(\rho(\widetilde{\varphi})\widetilde{\mathbf{v}}\otimes\widetilde{\mathbf{v}})^{\varepsilon}_i
         -\text{div}(\rho(\varphi)\mathbf{v}\otimes[\mathbf{v}]^\varepsilon)\Big)\,\mathrm{d}x\,\mathrm{d}t\notag\\
         &\qquad+\int_0^T\int_\Omega \chi_\tau\phi_\delta[\mathbf{v}]^{\varepsilon}
         \cdot\text{div}(\rho(\varphi)\mathbf{v}\otimes[\mathbf{v}]^\varepsilon)\,\mathrm{d}x\,\mathrm{d}t\notag\\
       &\quad=\int_0^T \int_\Omega \chi_\tau \phi_\delta \xi_0 [\mathbf{v}]^\varepsilon
       \cdot\text{div}\Big((\rho(\varphi) \mathbf{v}\otimes\mathbf{v})^\varepsilon-\rho(\varphi)\mathbf{v}\otimes[\mathbf{v}]^\varepsilon\Big)
       \,\mathrm{d}x\,\mathrm{d}t\notag\\
       &\qquad+\sum_{i=1}^k\int_0^T \int_\Omega \chi_\tau \phi_\delta \xi_i [\mathbf{v}]^\varepsilon
       \cdot\text{div} \Big((\rho(\widetilde{\varphi})\widetilde{\mathbf{v}}\otimes\widetilde{\mathbf{v}})_i^\varepsilon
       -\rho(\varphi)\mathbf{v} \otimes[\mathbf{v}]^\varepsilon\Big)\,\mathrm{d}x\,\mathrm{d}t\notag\\
       &\qquad+\int_0^T\int_\Omega \chi_\tau\phi_\delta[\mathbf{v}]^{\varepsilon}
       \cdot\text{div}(\rho(\varphi)\mathbf{v}\otimes[\mathbf{v}]^\varepsilon)\,\mathrm{d}x\,\mathrm{d}t\notag\\
       &\quad=-\int_0^T\int_\Omega \nabla(\chi_\tau \phi_\delta\xi_0[\mathbf{v}]^\varepsilon):
       \Big((\rho(\varphi) \mathbf{v}\otimes\mathbf{v})^\varepsilon-\rho(\varphi)\mathbf{v}\otimes[\mathbf{v}]^\varepsilon\Big)
       \,\mathrm{d}x\,\mathrm{d}t\notag\\
       &\qquad-\sum_{i=1}^k\int_0^T \int_\Omega \nabla(\chi_\tau \phi_\delta \xi_i [\mathbf{v}]^\varepsilon):
       \Big((\rho(\widetilde{\varphi}) \widetilde{\mathbf{v}}\otimes\widetilde{\mathbf{v}})_i^\varepsilon
       -\rho(\varphi)\mathbf{v}\otimes[\mathbf{v}]^\varepsilon\Big)\,\mathrm{d}x\,\mathrm{d}t\notag\\
       &\qquad+\int_0^T\int_\Omega \chi_\tau\phi_\delta[\mathbf{v}]^{\varepsilon}\cdot(\rho'(\varphi)(\mathbf{v}\cdot\nabla\varphi)
       [\mathbf{v}]^\varepsilon+\rho(\varphi)(\mathbf{v}\cdot\nabla)[\mathbf{v}]^\varepsilon)\,\mathrm{d}x\,\mathrm{d}t
       \notag\\
       &\quad=: \sum_{i=0}^k I_{\varepsilon,2}^{(i)}+J_{\varepsilon,2}.
       \notag
    \end{align}
For the term $I_{\varepsilon,2}^{(0)}$, it holds
    \begin{align}
        |I_{\varepsilon,2}^{(0)}|
        &\leq C\|[\mathbf{v}]^\varepsilon\|_{ L^q(\tau,T-\tau;\mathbf{W}^{1,r}(\Omega))}
        \|(\rho(\varphi)\mathbf{v}\otimes\mathbf{v})^\varepsilon
        -\rho(\varphi)\mathbf{v}\otimes[\mathbf{v}]^\varepsilon\|_{ L^{q'}(\tau,T-\tau;\mathbf{L}^{r'}(V_0))} \notag\\
        &\leq C\|\mathbf{v}\|_{L^q(0,T;\mathbf{W}^{1,r}_0(\Omega))}
        \|(\rho(\varphi)\mathbf{v}\otimes\mathbf{v})^\varepsilon
        -\rho(\varphi)\mathbf{v}\otimes[\mathbf{v}]^\varepsilon\|_{L^{q'}(\tau,T-\tau;\mathbf{L}^{r'}(V_0))}\notag\\
        &\leq C\|\nabla\mathbf{v}\|_{L^q(0,T;\mathbf{L}^{r}(\Omega))}
        \|(\rho(\varphi)\mathbf{v}\otimes\mathbf{v})^\varepsilon
        -\rho(\varphi)\mathbf{v}\otimes[\mathbf{v}]^\varepsilon\|_{L^{q'}(\tau,T-\tau;\mathbf{L}^{r'}(V_0))}\notag\\
        &\leq  C\|\nabla\mathbf{v}\|_{L^q(0,T;\mathbf{L}^{r}(\Omega))}
        \|(\rho(\varphi)\mathbf{v}\otimes\mathbf{v})^\varepsilon
        -\rho(\varphi)\mathbf{v}\otimes\mathbf{v}^\varepsilon\|_{L^{q'}(\tau,T-\tau;\mathbf{L}^{r'}(V_0))}\notag\\
        &\quad+ C\|\nabla\mathbf{v}\|_{L^q(0,T;\mathbf{L}^{r}(\Omega))}
        \|\rho(\varphi)\mathbf{v} \otimes(\mathbf{v}^\varepsilon-[\mathbf{v}]^\varepsilon)\|_{L^{q'}(\tau,T-\tau;\mathbf{L}^{r'}(V_0))}.\label{0203-I2-1}
    \end{align}
Noticing that $\rho(\varphi)\mathbf{v}$, $\mathbf{v}\in L^{2q'}(0,T;\mathbf{L}^{2r'}(\Omega))$, exploiting the properties of mollification, we can pass to the limit as $\varepsilon\to0$ in \eqref{0203-I2-1}
and obtain
$$\lim_{\varepsilon\to0}I_{\varepsilon,2}^{(0)}=0.$$
Similarly, we can conclude that
$$
\lim_{\varepsilon\to0}I_{\varepsilon,2}^{(i)}=0\quad \text{for}\ i=1,...,k.
$$
Concerning the term $J_{\varepsilon,2}$, we have
\begin{align*}
    J_{\varepsilon,2}
    &=\int_0^T\int_\Omega \chi_\tau\phi_\delta\rho'(\varphi)(\mathbf{v}\cdot\nabla\varphi)|[\mathbf{v}]^\varepsilon|^2
    +\chi_\tau\phi_\delta \rho(\varphi)\mathbf{v}\cdot\nabla\Big(\frac{1}{2}|[\mathbf{v}]^\varepsilon|^2\Big)\,\mathrm{d}x\,\mathrm{d}t
    \\
    &=\frac{1}{2}\int_0^T\int_\Omega\chi_\tau\phi_\delta\rho'(\varphi)(\mathbf{v}\cdot\nabla\varphi)|[\mathbf{v}]^\varepsilon|^2\,\mathrm{d}x\,\mathrm{d}t
    \\
    &\quad -\frac{1}{2}\int_0^T\int_\Omega \chi_\tau\rho(\varphi)(\mathbf{v}\cdot\nabla\phi_\delta)|[\mathbf{v}]^\varepsilon|^2\,\mathrm{d}x\,\mathrm{d}t.
\end{align*}
By the regularity properties
{$\mathbf{v}\in L^4(0,T;\mathbf{L}^4(\Omega))$} and $\nabla\varphi\in L^\infty(\tau/2,T;H^1(\Omega))$,
we can pass to the limit as $\varepsilon\to 0$ and conclude that
\begin{align}
    \lim_{\varepsilon\to0}J_{\varepsilon,2}
    &=\frac{1}{2}\int_0^T\int_\Omega\chi_\tau\phi_\delta\rho'(\varphi)(\mathbf{v}\cdot\nabla\varphi)|\mathbf{v}|^2\,\mathrm{d}x\,\mathrm{d}t
    \notag \\
    &\quad -\frac{1}{2}\int_0^T\int_\Omega\chi_\tau\rho(\varphi)(\mathbf{v}\cdot\nabla\phi_\delta)|\mathbf{v}|^2\,\mathrm{d}x\,\mathrm{d}t.
    \notag
\end{align}

Now for the term involving the flux $\mathbf{J}$, we find that
\begin{align}
    &\int_0^T\int_\Omega\chi_\tau \phi_\delta[\mathbf{v}]^{\varepsilon}\cdot\Big(\xi_0\mathrm{div}(\mathbf{v}\otimes\mathbf{J})^{\varepsilon}
    +\sum_{i=1}^k \xi_i\mathrm{div}(\widetilde{\mathbf{v}}\otimes\widetilde{\mathbf{J}})_i^{\varepsilon}\Big)\,\mathrm{d}x\,\mathrm{d}t\notag\\
    &\quad=\int_0^T\int_\Omega \chi_\tau\phi_\delta[\mathbf{v}]^{\varepsilon}\cdot\Big(\xi_0\mathrm{div}(\mathbf{v}\otimes\mathbf{J})^{\varepsilon}
    +\sum_{i=1}^k \xi_i\mathrm{div}(\widetilde{\mathbf{v}}\otimes\widetilde{\mathbf{J}})_i^{\varepsilon}
    -\text{div}([\mathbf{v}]^\varepsilon\otimes\mathbf{J})\Big)\,\mathrm{d}x\,\mathrm{d}t\notag\\
    &\qquad+\int_0^T\int_\Omega \chi_\tau\phi_\delta[\mathbf{v}]^\varepsilon
    \cdot\text{div}([\mathbf{v}]^\varepsilon\otimes\mathbf{J})\,\mathrm{d}x\,\mathrm{d}t\notag\\
    &\quad=\int_0^T\int_\Omega\chi_\tau \phi_\delta\xi_0[\mathbf{v}]^\varepsilon \cdot
    \text{div}\Big(({\mathbf{v}} \otimes{\mathbf{J}})^\varepsilon-[\mathbf{v}]^\varepsilon\otimes\mathbf{J}\Big)\,\mathrm{d}x\,\mathrm{d}t\notag\\
    &\qquad+\sum_{i=1}^k\int_0^T \int_\Omega\chi_\tau \phi_\delta\xi_i [\mathbf{v}]^\varepsilon \cdot
    \text{div}\Big((\widetilde{\mathbf{v}} \otimes\widetilde{\mathbf{J}})^\varepsilon_i-[\mathbf{v}]^\varepsilon\otimes\mathbf{J}\Big)\,\mathrm{d}x\,\mathrm{d}t\notag\\
    &\qquad+\int_0^T\int_\Omega \chi_\tau\phi_\delta[\mathbf{v}]^\varepsilon\cdot
    \text{div}([\mathbf{v}]^\varepsilon\otimes\mathbf{J})\,\mathrm{d}x\,\mathrm{d}t\notag\\
    &\quad =:\sum_{i=0}^k I^{(i)}_{\varepsilon,3}+J_{\varepsilon,3}.
    \notag
\end{align}
For the term $I_{\varepsilon,3}^{(0)}$, using the property {$\mathbf{v}\in L^4(0,T;\mathbf{L}^4(\Omega))$} and
$$\mathbf{J}=-\rho'(\varphi)\nabla\mu\in  L^\infty(\tau/2,T;\mathbf{L}^2(\Omega))\cap L^2(\tau/2,T;\mathbf{H}^2(\Omega))
\hookrightarrow L^4(\tau/2,T;\mathbf{L}^4(\Omega)),$$
similarly to \eqref{0203-I2-1}, we can apply \eqref{local-convergence-1} and Lemma \ref{Commutator-lemma-3} to get
\begin{align*}
    |I_{\varepsilon,3}^{(0)}|
    &=\Big|\int_0^T\int_\Omega\nabla(\chi_\tau \phi_\delta\xi_0[\mathbf{v}]^\varepsilon):
    \Big(({\mathbf{v}} \otimes{\mathbf{J}})^\varepsilon-[\mathbf{v}]^\varepsilon\otimes\mathbf{J}\Big)\,\mathrm{d}x\,\mathrm{d}t\Big|\\
    &\leq C\|\nabla\mathbf{v}\|_{L^2(0,T;\mathbf{L}^2(\Omega))}
    \|(\mathbf{v}\otimes\mathbf{J})^\varepsilon-\mathbf{v}^\varepsilon\otimes\mathbf{J}\|_{L^2(\tau,T-\tau;\mathbf{L}^2(V_0))}\\
    &\quad+C\|\nabla\mathbf{v} \|_{L^2(0,T;\mathbf{L}^2(\Omega))}
    \|([\mathbf{v}]^\varepsilon-\mathbf{v}^\varepsilon)\otimes\mathbf{J}\|_{L^2(\tau,T-\tau;\mathbf{L}^2(V_0))}\\
    &\to 0\quad\text{as }\ \varepsilon\to0.
\end{align*}
Similarly, we have
$$
\lim_{\varepsilon\to0}I_{\varepsilon,3}^{(i)}=0\quad \text{for}\ i=1,...,k.
$$
Finally, for the term $J_{\varepsilon,3}$, we find
\begin{align}
    J_{\varepsilon,3}
    &=-\int_0^T \int_\Omega \chi_\tau \phi_\delta \rho'(\varphi)[\mathbf{v}]^\varepsilon \cdot((\nabla\mu\cdot\nabla)[\mathbf{v}]^\varepsilon
    +\Delta\mu[\mathbf{v}]^\varepsilon) \,\mathrm{d}x\,\mathrm{d}t
    \notag\\
    &=-\frac{1}{2}\int_0^T\int_\Omega\chi_\tau\phi_\delta \rho'(\varphi)\nabla\mu\cdot\nabla(|[\mathbf{v}]^\varepsilon|^2)\,\mathrm{d}x\,\mathrm{d}t
    \notag\\
    &\quad -\int_0^T\int_\Omega\chi_\tau\phi_\delta \rho'(\varphi)\Delta\mu|[\mathbf{v}]^\varepsilon|^2\,\mathrm{d}x\,\mathrm{d}t
    \notag\\
    &=\frac{1}{2}\int_0^T\int_\Omega\chi_\tau\rho'(\varphi)(\nabla\mu\cdot\nabla\phi_\delta)|[\mathbf{v}]^\varepsilon|^2\,\mathrm{d}x\,\mathrm{d}t
    \notag\\
    &\quad -\frac{1}{2}\int_0^T\int_\Omega\chi_\tau\phi_\delta \rho'(\varphi)\Delta\mu|[\mathbf{v}]^\varepsilon|^2\,\mathrm{d}x\,\mathrm{d}t.
    \notag
\end{align}
By the regularity {$\mathbf{v}\in L^4(0,T;\mathbf{L}^4(\Omega))$} and $\mu\in L^2(\tau/2,T;H^3(\Omega))$, we can conclude
\begin{align}
    \lim_{\varepsilon\to0}J_{\varepsilon,3}
    & =\frac{1}{2}\int_0^T\int_\Omega\chi_\tau\rho'(\varphi)(\nabla\mu\cdot\nabla\phi_\delta)|\mathbf{v}|^2\,\mathrm{d}x\,\mathrm{d}t
    \notag \\
    &\quad -\frac{1}{2}\int_0^T\int_\Omega\chi_\tau\phi_\delta \rho'(\varphi)\Delta\mu|\mathbf{v}|^2\,\mathrm{d}x\,\mathrm{d}t.
    \notag
\end{align}

Collecting the above results and using the strong form of the Cahn--Hilliard equation \eqref{strong1}, we can conclude \eqref{epsilon-limit-1}.
This completes the proof of Lemma \ref{epsilon-limit-lemma-1}.
\end{proof}

\begin{lem}
\label{epsilon-limit-lemma-2}
    For fixed $\tau$ and $\delta$, the fourth term on the left-hand side of  \eqref{mollify-eq2} satisfies
    \begin{align}
        &\lim_{\varepsilon\to0}-\int_0^T\int_\Omega\chi_\tau\phi_\delta[\mathbf{v}]^{\varepsilon}\cdot
        \Big(\xi_0 \mathrm{div}(2\nu(\varphi)\mathbb{D}\mathbf{v})^{\varepsilon}
        +\sum_{i=1}^k\xi_i \mathrm{div}(2\nu(\widetilde{\varphi})\mathbb{D}\widetilde{\mathbf{v}})_i^{\varepsilon}\Big)\,\mathrm{d}x\,\mathrm{d}t\notag\\
        &\quad=2\int_0^T \int_\Omega \chi_\tau\phi_\delta\nu(\varphi)|\mathbb{D}\mathbf{v}|^2\,\mathrm{d}x\,\mathrm{d}t
        +2\int_0^T \int_\Omega \chi_\tau\nu(\varphi)(\mathbb{D}\mathbf{v}\,\mathbf{v})\cdot\nabla\phi_\delta\,\mathrm{d}x\,\mathrm{d}t.\label{epsilon-limit2}
    \end{align}
\end{lem}

\begin{proof}
   It is easy to check that
    \begin{align}
        &-\int_0^T\int_\Omega\chi_\tau\phi_\delta[\mathbf{v}]^{\varepsilon}
        \cdot\Big(\xi_0 \mathrm{div}(2\nu(\varphi)\mathbb{D}\mathbf{v})^{\varepsilon}
        +\sum_{i=1}^k\xi_i \mathrm{div}(2\nu(\widetilde{\varphi})\mathbb{D}\widetilde{\mathbf{v}})_i^{\varepsilon}\Big)\,\mathrm{d}x\,\mathrm{d}t\notag\\
        &\quad=-\int_0^T\int_\Omega\chi_\tau\phi_\delta[\mathbf{v}]^{\varepsilon}
        \cdot\Big(\xi_0 \mathrm{div}(2\nu(\varphi)\mathbb{D}\mathbf{v})^{\varepsilon}
        +\sum_{i=1}^k\xi_i \mathrm{div}(2\nu(\widetilde{\varphi})\mathbb{D}\widetilde{\mathbf{v}})_i^{\varepsilon}
        -\text{div}(2\nu(\varphi)[\mathbb{D}\mathbf{v}]^{\varepsilon})\Big)\,\mathrm{d}x\,\mathrm{d}t\notag\\
        &\qquad-\int_0^T\int_\Omega\chi_\tau\phi_\delta[\mathbf{v}]^{\varepsilon}
        \cdot\text{div}(2\nu(\varphi)[\mathbb{D}\mathbf{v}]^{\varepsilon})\,\mathrm{d}x\,\mathrm{d}t\notag\\
        &\quad=-\int_0^T\int_\Omega\chi_\tau\phi_\delta\xi_0[\mathbf{v}]^{\varepsilon}
        \cdot\mathrm{div}\Big( (2\nu(\varphi)\mathbb{D}\mathbf{v})^{\varepsilon}
        -2\nu(\varphi)[\mathbb{D}\mathbf{v}]^{\varepsilon}\Big)\,\mathrm{d}x\,\mathrm{d}t\notag\\
        &\qquad-\sum_{i=1}^k\int_0^T\int_\Omega\chi_\tau\phi_\delta\xi_i[\mathbf{v}]^{\varepsilon}
        \cdot\mathrm{div}\Big( (2\nu(\widetilde{\varphi})\mathbb{D}\widetilde{\mathbf{v}})^{\varepsilon}_i
        -2\nu(\varphi)[\mathbb{D}\mathbf{v}]^{\varepsilon}\Big)\,\mathrm{d}x\,\mathrm{d}t\notag\\
        &\qquad+\left(2\int_0^T \int_\Omega \chi_\tau\phi_\delta\nu(\varphi)[\mathbb{D}\mathbf{v}]^{\varepsilon}:\mathbb{D}[\mathbf{v}]^\varepsilon\,\mathrm{d}x\,\mathrm{d}t
        +2\int_0^T \int_\Omega \chi_\tau\nu(\varphi)([\mathbb{D}\mathbf{v}]^{\varepsilon} [\mathbf{v}]^{\varepsilon})\cdot\nabla\phi_\delta\,\mathrm{d}x\,\mathrm{d}t\right)\notag\\
        &\quad=:\sum_{i=0}^k I^{(i)}_{\varepsilon,4}+J_{\varepsilon,4}.\label{term4}
    \end{align}
   For the term $I_{\varepsilon,4}^{(0)}$, similarly to \eqref{0203-I2-1}, we have
    \begin{align}
        |I_{\varepsilon,4}^{(0)}|
        &= \Big|\int_0^T \int_\Omega \nabla(\chi_\tau \phi_\delta \xi_0[\mathbf{v}]^{\varepsilon})
        :\Big((2\nu(\varphi) \mathbb{D}\mathbf{v})^{\varepsilon}-2\nu(\varphi)[\mathbb{D}\mathbf{v}]^{\varepsilon}\Big)\,\mathrm{d}x\,\mathrm{d}t\Big|\notag\\
        &\leq C\|\nabla\mathbf{v}\|_{ L^2(0,T;\mathbf{L}^2(\Omega))}\|(2\nu(\varphi)\mathbb{D}\mathbf{v})^{\varepsilon}
        -2\nu(\varphi)[\mathbb{D}\mathbf{v}]^{\varepsilon}\|_{L^{2}(\tau,T-\tau;\mathbf{L}^{2}(V_0))}\notag\\
        &\leq C\|\nabla\mathbf{v}\|_{L^2(0,T;\mathbf{L}^2(\Omega))}\|(2\nu(\varphi)\mathbb{D}\mathbf{v})^{\varepsilon}
        -2\nu(\varphi)(\mathbb{D}\mathbf{v})^{\varepsilon}\|_{L^{2}(\tau,T-\tau;\mathbf{L}^{2}(V_0))}\notag\\
        &\quad+C\|\nabla\mathbf{v}\|_{L^2(0,T;\mathbf{L}^2(\Omega))}\|2\nu(\varphi)((\mathbb{D}\mathbf{v})^{\varepsilon}
        -[\mathbb{D}\mathbf{v}]^{\varepsilon})\|_{L^{2}(\tau,T-\tau;\mathbf{L}^{2}(V_0))}.\notag
    \end{align}
    Using the facts $\nu(\varphi)\in L^\infty(0,T;L^\infty(\Omega))$ and $\nabla\mathbf{v}\in L^2(0,T;\mathbf{L}^2(\Omega))$,
    we can conclude that
    $$
    \lim_{\varepsilon\to0}I_{\varepsilon,4}^{(0)}=0.
    $$
    Similarly, it holds
    $$
    \lim_{\varepsilon\to0}I_{\varepsilon,4}^{(i)}=0\quad \text{for}\ i=1,...,k.
    $$
    Besides, since $\mathbf{v}\in L^2(0,T;\mathbf{H}_\sigma^1(\Omega))$,
    we can pass to the limit as $\varepsilon\to0$ for $J_{\varepsilon,4}$ and obtain
    \begin{align}
        \lim_{\varepsilon\to0}J_{\varepsilon,4}
        = 2\int_0^T \int_\Omega \chi_\tau\phi_\delta\nu(\varphi)|\mathbb{D}\mathbf{v}|^2\,\mathrm{d}x\,\mathrm{d}t
        +2\int_0^T \int_\Omega \chi_\tau(\nu(\varphi)\mathbb{D}\mathbf{v} \, \mathbf{v})\cdot\nabla\phi_\delta\,\mathrm{d}x\,\mathrm{d}t.\notag
    \end{align}
   Hence, we infer from \eqref{term4} that \eqref{epsilon-limit2} holds. This completes the proof of Lemma \ref{epsilon-limit-lemma-2}.
\end{proof}

\begin{lem}
\label{epsilon-limit-lemma-3}
    For fixed $\tau$ and $\delta$, the fifth term on the left-hand side of \eqref{mollify-eq2} satisfies
    \begin{align}
       & \lim_{\varepsilon\to0}-\int_0^T\int_\Omega\chi_\tau\phi_\delta[\mathbf{v}]^{\varepsilon}\cdot\Big(\xi_0(\mu\nabla\varphi)^{\varepsilon}
       +\sum_{i=1}^k\xi_i (\widetilde{\mu}\nabla\widetilde{\varphi})_i^{\varepsilon}\Big)\,\mathrm{d}x\,\mathrm{d}t
       \notag \\
       &\quad =-\int_0^T\int_\Omega \chi_\tau \phi_\delta \mu\mathbf{v}\cdot\nabla\varphi\,\mathrm{d}x\,\mathrm{d}t.
       \label{epsilon-limit-3}
    \end{align}
\end{lem}
\begin{proof}
From the regularity properties $\mu\in L^\infty(\tau/2,T;H^1(\Omega))$ and $\varphi\in L^\infty(\tau/2,T;H^2(\Omega))$,
we see that
\[\mu\nabla\varphi\in L^2(\tau/2,T;\mathbf{L}^2(\Omega)).\]
Then, by \eqref{mollify-convergence}, we can pass to the limit as $\varepsilon\to0$ on the left-hand side of \eqref{epsilon-limit-3} and obtain
\begin{align}
     & \lim_{\varepsilon\to0}-\int_0^T\int_\Omega\chi_\tau\phi_\delta[\mathbf{v}]^{\varepsilon}\cdot\Big(\xi_0(\mu\nabla\varphi)^{\varepsilon}
     +\sum_{i=1}^k\xi_i (\widetilde{\mu}\nabla\widetilde{\varphi})_i^{\varepsilon}\Big)\,\mathrm{d}x\,\mathrm{d}t\notag\\
     &\quad=\lim_{\varepsilon\to0}-\int_0^T\int_\Omega\chi_\tau\phi_\delta[\mathbf{v}]^{\varepsilon}\cdot[\mu\nabla\varphi]^\varepsilon\,\mathrm{d}x\,\mathrm{d}t
     \notag\\
     &\quad =-\int_0^T\int_\Omega \chi_\tau \phi_\delta \mu\mathbf{v}\cdot\nabla\varphi\,\mathrm{d}x\,\mathrm{d}t.\notag
\end{align}
This completes the proof of Lemma \ref{epsilon-limit-lemma-3}.
\end{proof}

\begin{lem}
\label{epsilon-limit-lemma-4}
 For fixed $\tau$ and $\delta$, the sixth term on the left-hand side of \eqref{mollify-eq2} satisfies
    \begin{align}
&\lim_{\varepsilon\to0}\int_0^T\int_\Omega\chi_\tau\phi_\delta[\mathbf{v}]^{\varepsilon}
\cdot\Big(\xi_0\nabla p^{\varepsilon}+\sum_{i=1}^k \xi_i \nabla \widetilde{p}^{\varepsilon}_i\Big)\,\mathrm{d}x\,\mathrm{d}t
\notag\\
&\quad =-\int_0^T\int_\Omega\chi_\tau(\mathbf{v}\cdot\nabla\phi_\delta)p\,\mathrm{d}x\,\mathrm{d}t.
\label{epsilon-limit-4}
    \end{align}
\end{lem}
\begin{proof}
  A direct calculation yields that
    \begin{align}
        &\int_0^T\int_\Omega \chi_\tau\phi_\delta[\mathbf{v}]^{\varepsilon}
        \cdot\Big(\xi_0\nabla p^{\varepsilon}+\sum_{i=1}^k \xi_i \nabla \widetilde{p}^{\varepsilon}_i\Big)\,\mathrm{d}x\,\mathrm{d}t\notag\\
        &\quad=\int_0^T\int_\Omega \chi_\tau\phi_\delta[\mathbf{v}]^{\varepsilon}
        \cdot\Big(\xi_0\nabla p^{\varepsilon}+\sum_{i=1}^k \xi_i \nabla \widetilde{p}^{\varepsilon}_i-\nabla[p]^{\varepsilon}\Big)\,\mathrm{d}x\,\mathrm{d}t
        \notag \\
        &\qquad +\int_0^T\int_\Omega \chi_\tau\phi_\delta[\mathbf{v}]^{\varepsilon}\cdot\nabla[p]^{\varepsilon}\,\mathrm{d}x\,\mathrm{d}t\notag\\
        &\quad=\int_0^T\int_\Omega \chi_\tau\phi_\delta[\mathbf{v}]^{\varepsilon}
        \cdot\Big(\xi_0\nabla p^{\varepsilon}+\sum_{i=1}^k \xi_i \nabla \widetilde{p}^{\varepsilon}_i-\nabla[p]^{\varepsilon}\Big)\,\mathrm{d}x\,\mathrm{d}t
        \notag\\
        &\qquad -\int_0^T\int_\Omega\chi_\tau\text{div}([\mathbf{v}]^{\varepsilon}\phi_\delta)[p]^{\varepsilon}\,\mathrm{d}x\,\mathrm{d}t\notag\\
        &\quad=:I_{\varepsilon,5} +J_{\varepsilon,5}.\label{term6}
    \end{align}
Using the regularity properties
\begin{align}
    p\in L^{q'}(\tau/2,T;L^{r'}(\Omega))
\quad\text{and} \quad\mathbf{v}\in L^q(0,T;\mathbf{W}_0^{1,r}(\Omega)),\label{0207-1}
\end{align}
we deduce that
\begin{align}
    |I_{\varepsilon,5}|
    &\leq \Big|\int_0^T\int_\Omega\chi_\tau\phi_\delta\xi_0[\mathbf{v}]^{\varepsilon}
    \cdot\nabla(p^{\varepsilon}-[p]^{\varepsilon})\,\mathrm{d}x\,\mathrm{d}t\Big|
    \notag\\
    &\quad +\sum_{i=1}^k\Big|\int_0^T\int_\Omega \chi_\tau\phi_\delta  \xi_i [\mathbf{v}]^{\varepsilon}
    \cdot \nabla(\widetilde{p}_i^{\varepsilon}-[p]^{\varepsilon})\,\mathrm{d}x\,\mathrm{d}t\Big|\notag\\
    &=\Big|\int_0^T\int_\Omega\text{div}(\chi_\tau\phi_\delta\xi_0[\mathbf{v}]^{\varepsilon})(p^{\varepsilon}-[p]^{\varepsilon})\,\mathrm{d}x\,\mathrm{d}t\Big|
    \notag\\
    &\quad +\sum_{i=1}^k\Big|\int_0^T \int_\Omega\text{div}(\chi_\tau\phi_\delta\xi_i[\mathbf{v}]^{\varepsilon} )
    (\widetilde{p}_i^{\varepsilon}-[p]^{\varepsilon})\,\mathrm{d}x\,\mathrm{d}t\Big|\notag\\
    &\leq C\int_{\tau}^{T-\tau}\|\text{div}(\chi_\tau\phi_\delta\xi_0[\mathbf{v}]^{\varepsilon})\|_{L^{r}}
    \|(p^{\varepsilon}-[p]^{\varepsilon})\|_{L^{r'}(V_0)}\,\mathrm{d}t\notag\\
    &\quad+C\sum_{i=1}^k\int_{\tau}^{T-\tau}\|\text{div}(\chi_\tau\phi_\delta\xi_i[\mathbf{v}]^{\varepsilon})\|_{L^{r}}
    \|(\widetilde{p}_i^{\varepsilon}-[p]^{\varepsilon})\|_{L^{r'}(V_i)}\,\mathrm{d}t\notag\\
    &\to 0\quad\text{as }\ \varepsilon\to0.\label{I6}
\end{align}
For the term $J_{\varepsilon,5}$, by \eqref{0207-1} and \eqref{mollify-convergence}, we obtain
\begin{align}
    \lim_{\varepsilon\to0}J_{\varepsilon,5}
    =-\int_0^T\int_\Omega\chi_\tau \text{div}(\mathbf{v}\phi_\delta)p\,\mathrm{d}x\,\mathrm{d}t
    = -\int_0^T\int_\Omega\chi_\tau (\mathbf{v}\cdot\nabla\phi_\delta)p\,\mathrm{d}x\,\mathrm{d}t.\label{J5}
\end{align}
From \eqref{term6}, \eqref{I6} and \eqref{J5},
we can conclude \eqref{epsilon-limit-4}. This completes the proof of Lemma \ref{epsilon-limit-lemma-4}.
\end{proof}

%%%%%%%%%%%%%%%%%%%%%%%%%%%%%%%%%%%
\subsection{\texorpdfstring{Passage to the limit as $\delta,\tau\to0$}{Passage to the limit as \delta, \tau \to 0}}
\label{sec:3.3}
%\label{energy-limit}

Thanks to Lemmas \ref{epsilon-limit-lemma-1}--\ref{epsilon-limit-lemma-4},
we can pass to the limit as $\varepsilon\to0$ in \eqref{mollify-eq2}. Adding the resultant to \eqref{mollify-eq}, we obtain
\begin{align}
    &-\frac{1}{2}\int_0^T\int_\Omega\chi_\tau'\phi_\delta \rho(\varphi)|\mathbf{v}|^2\,\mathrm{d}x\,\mathrm{d}t
    +2\int_0^T \int_\Omega \chi_\tau\phi_\delta\nu(\varphi)|\mathbb{D}\mathbf{v}|^2\,\mathrm{d}x\,\mathrm{d}t\notag\\
    &\qquad-\frac{1}{2}\int_0^T\int_\Omega\chi_\tau\rho(\varphi)(\mathbf{v}\cdot\nabla\phi_\delta)|\mathbf{v}|^2\,\mathrm{d}x\,\mathrm{d}t
    \notag\\
    &\qquad +\frac{1}{2}\int_0^T\int_\Omega\chi_\tau\rho'(\varphi)(\nabla\mu\cdot\nabla\phi_\delta)|\mathbf{v}|^2\,\mathrm{d}x\,\mathrm{d}t\notag\\
    &\qquad+2\int_0^T \int_\Omega \chi_\tau\nu(\varphi)(\mathbb{D}\mathbf{v}\,\mathbf{v})\cdot\nabla\phi_\delta\,\mathrm{d}x\,\mathrm{d}t
        -\int_0^T\int_\Omega\chi_\tau(\mathbf{v}\cdot\nabla\phi_\delta)p\,\mathrm{d}x\,\mathrm{d}t\notag\\
    &\qquad+\int_0^{T}\int_\Omega\chi_\tau\phi_\delta(-\Delta\varphi+\Psi'(\varphi))\partial_t\varphi\,\mathrm{d}x\,\mathrm{d}t
        -\int_0^T\int_\Omega\chi_\tau\phi_\delta\mu\Delta\mu\,\mathrm{d}x\,\mathrm{d}t
       =0.
    \label{epsilon-limit-form}
\end{align}
With identity \eqref{epsilon-limit-form} in hand, we are in a position to prove Theorem \ref{criterion}.
\smallskip

\noindent\textbf{Proof of Theorem \ref{criterion}.} The proof consists of two limiting procedures.
\smallskip

\emph{Step 1: Passage to the limit as $\delta\to0$.}
By the Lebesgue dominated convergence theorem, we have
\begin{align}
    &\lim_{\delta\to0}\left(-\frac{1}{2}\int_0^T\int_\Omega\chi_\tau'\phi_\delta \rho(\varphi)|\mathbf{v}|^2\,\mathrm{d}x\,\mathrm{d}t
    +2\int_0^T \int_\Omega \chi_\tau\phi_\delta\nu(\varphi)|\mathbb{D}\mathbf{v}|^2\,\mathrm{d}x\,\mathrm{d}t\right.\notag\\
    &\qquad\left.+\int_0^{T}\int_\Omega \chi_\tau\phi_\delta(-\Delta\varphi+\Psi'(\varphi))\partial_t\varphi\,\mathrm{d}x\,\mathrm{d}t
    -\int_0^T\int_\Omega\chi_\tau\phi_\delta\mu\Delta\mu\,\mathrm{d}x\,\mathrm{d}t\right)\notag\\
    &\quad=-\frac{1}{2}\int_0^T\int_\Omega\chi_\tau'\rho(\varphi)|\mathbf{v}|^2\,\mathrm{d}x\,\mathrm{d}t
    +2\int_0^T \int_\Omega \chi_\tau\nu(\varphi)|\mathbb{D}\mathbf{v}|^2\,\mathrm{d}x\,\mathrm{d}t\notag\\
    &\qquad+\int_0^{T}\int_\Omega\chi_\tau(-\Delta\varphi+\Psi'(\varphi))\partial_t\varphi\,\mathrm{d}x\,\mathrm{d}t
    -\int_0^T\int_\Omega\chi_\tau\mu\Delta\mu\,\mathrm{d}x\,\mathrm{d}t\notag\\
    &\quad=-\frac{1}{2}\int_0^T\int_\Omega\chi_\tau'\rho(\varphi)|\mathbf{v}|^2\,\mathrm{d}x\,\mathrm{d}t
    +2\int_0^T \int_\Omega \chi_\tau\nu(\varphi)|\mathbb{D}\mathbf{v}|^2\,\mathrm{d}x\,\mathrm{d}t\notag\\
    &\qquad-\int_0^{T}\int_\Omega\chi_\tau'\Big(\frac{|\nabla\varphi|^2}{2}+\Psi(\varphi)\Big)\,\mathrm{d}x\,\mathrm{d}t
    +\int_0^T\int_\Omega\chi_\tau |\nabla\mu|^2\,\mathrm{d}x\,\mathrm{d}t,\label{delta-limit-1}
\end{align}
where we use the chain rule of maximal monotone operators (see \cite[Chapter 4, Lemma 4.4]{Barbu}) to get
\[\int_0^T\int_\Omega\chi_\tau \Psi_0'(\varphi)\partial_t\varphi\,\mathrm{d}x\,\mathrm{d}t=\int_0^T\chi_\tau\frac{\mathrm{d}}{\mathrm{d}t}\int_\Omega\Psi_0(\varphi)\,\mathrm{d}x\,\mathrm{d}t=-\int_0^T\int_\Omega\chi_\tau'\Psi_0(\varphi)\,\mathrm{d}x\,\mathrm{d}t.\]
Moreover, by Lemma \ref{Hardy}, we find
\begin{align}
    &\Big|\frac{1}{2}\int_0^T\int_\Omega\chi_\tau\rho(\varphi)(\mathbf{v}\cdot\nabla\phi_\delta)|\mathbf{v}|^2\,\mathrm{d}x\,\mathrm{d}t\Big|
    +\Big|\frac{1}{2}\int_0^T\int_\Omega\chi_\tau\rho'(\varphi)(\nabla\mu\cdot\nabla\phi_\delta)|\mathbf{v}|^2\,\mathrm{d}x\,\mathrm{d}t\Big|\notag\\
        &\qquad+\Big|2\int_0^T \int_\Omega \chi_\tau\nu(\varphi)(\mathbb{D}\mathbf{v}\,\mathbf{v})\cdot\nabla\phi_\delta\,\mathrm{d}x\,\mathrm{d}t\Big|
        +\Big|\int_0^T\int_\Omega \chi_\tau(\mathbf{v}\cdot\nabla\phi_\delta)p\,\mathrm{d}x\,\mathrm{d}t\Big|\notag\\
        &\quad\leq C\Big\|\frac{\mathbf{v}}{\mathrm{dist}(x,\partial\Omega)}\Big\|_{L^q(0,T;\mathbf{L}^r(\Omega))}
        \|\mathbf{v}\|_{L^{2q'}(0,T;\mathbf{L}^{2r'}(\Omega_\delta))}^2\notag\\
          &\qquad +C\Big\|\frac{\mathbf{v}}{\mathrm{dist}(x,\partial\Omega)}\Big\|_{L^q(0,T;\mathbf{L}^r(\Omega))}
          \|\nabla\mu\|_{L^2(0,T; \mathbf{L}^2(\Omega_\delta))}
          \|\mathbf{v}\|_{L^\frac{2q}{q-2}(0,T;\mathbf{L}^{\frac{2r}{r-2}}(\Omega_\delta))}\notag\\
         &\qquad+\Big\|\frac{\mathbf{v}}{\mathrm{dist}(x,\partial\Omega)}\Big\|_{L^{q}(0,T;\mathbf{L}^r(\Omega))}
         \|\mathbb{D}\mathbf{v} \|_{L^{q'}(0,T;\mathbf{L}^{r'}(\Omega_\delta))}\notag\\
        &\qquad+C\Big\|\frac{\mathbf{v}}{\mathrm{dist}(x,\partial\Omega)}\Big\|_{L^{q}(0,T;\mathbf{L}^r(\Omega))}
        \|p\|_{L^{q'}(\tau,T-\tau;{L}^{r'}(\Omega_\delta))}\notag\\
        &\quad\leq C\|\mathbf{v}\|_{L^q(0,T;\mathbf{W}_0^{1,r}(\Omega))}\mathcal{F}_\delta
        \to0\quad\text{as }\ \delta\to0,
        \label{delta-limit-2}
\end{align}
where $$\Omega_\delta:=\{x\in\Omega:\text{dist}(x,\partial\Omega)<\delta\}$$
and the function $\mathcal{F}_\delta$ is given by
\begin{align*}
    \mathcal{F}_\delta:
    &=\|\mathbf{v}\|_{L^{2q'}(0,T;\mathbf{L}^{2r'}(\Omega_\delta))}^2
    +\|\nabla\mu\|_{L^2(0,T; \mathbf{L}^2(\Omega_\delta))}
    \|\mathbf{v}\|_{L^\frac{2q}{q-2}(0,T;\mathbf{L}^\frac{2r}{r-2}(\Omega_\delta))}\\
    &\quad+\|\mathbb{D}\mathbf{v}\|_{L^{q'}(0,T;\mathbf{L}^{r'}(\Omega_\delta))}
    +\|p\|_{L^{q'}(\tau,T-\tau;{L}^{r'}(\Omega_\delta))}.
\end{align*}
According to \eqref{delta-limit-1} and \eqref{delta-limit-2},
we can pass to the limit as $\delta\to0$ in \eqref{epsilon-limit-form} and obtain
\begin{align}
    &\frac{1}{2}\int_0^T\int_\Omega\chi_\tau' \rho(\varphi)|\mathbf{v}|^2\,\mathrm{d}x\,\mathrm{d}t
    +\int_0^T\int_\Omega\chi_\tau' \Big(\frac{1}{2}|\nabla\varphi|^2+\Psi(\varphi)\Big)\,\mathrm{d}x\,\mathrm{d}t\notag\\
    &\quad-2\int_0^T \int_\Omega \chi_\tau\nu(\varphi)|\mathbb{D}\mathbf{v}|^2\,\mathrm{d}x\,\mathrm{d}t
    -\int_0^T\int_\Omega\chi_\tau |\nabla\mu|^2\,\mathrm{d}x\,\mathrm{d}t=0.\label{delta-limit-form}
\end{align}

\emph{Step 2: Passage to the limit as $\tau\to0$.}
The Lebesgue dominated convergence theorem yields
\begin{align}
    &\lim_{\tau\to0}\Big(-2\int_0^T \int_\Omega \chi_\tau\nu(\varphi)|\mathbb{D}\mathbf{v}|^2\,\mathrm{d}x\,\mathrm{d}t
    -\int_0^T\int_\Omega\chi_\tau |\nabla\mu|^2\,\mathrm{d}x\,\mathrm{d}t\Big)\notag\\
    &\quad=-2\int_0^T \int_\Omega \nu(\varphi)|\mathbb{D}\mathbf{v}|^2\,\mathrm{d}x\,\mathrm{d}t
    -\int_0^T\int_\Omega|\nabla\mu|^2\,\mathrm{d}x\,\mathrm{d}t.\label{tau-limit-1}
\end{align}
Moreover, we have
\begin{align}
    &\frac{1}{2}\int_0^T\int_\Omega\chi_\tau' \rho(\varphi)|\mathbf{v}|^2\,\mathrm{d}x\,\mathrm{d}t
    +\int_0^T\int_\Omega\chi_\tau' \Big(\frac{1}{2}|\nabla\varphi|^2+\Psi(\varphi)\Big)\,\mathrm{d}x\,\mathrm{d}t\notag\\
    &\quad=\frac{1}{\tau}\int_\tau^{2\tau}\int_\Omega\rho(\varphi)\frac{|\mathbf{v}|^2}{2}\,\mathrm{d}x\,\mathrm{d}t
    +\frac{1}{\tau}\int_{\tau}^{2\tau}\int_\Omega \Big(\frac{1}{2}|\nabla\varphi|^2+\Psi(\varphi)\Big)\,\mathrm{d}x\,\mathrm{d}t\notag\\
    &\qquad-\frac{1}{\tau}\int_{T-2\tau}^{T-\tau}\int_\Omega\rho(\varphi)\frac{|\mathbf{v}|^2}{2}\,\mathrm{d}x\,\mathrm{d}t
    -\frac{1}{\tau}\int_{T-2\tau}^{T-\tau}\int_\Omega \Big(\frac{1}{2}|\nabla\varphi|^2+\Psi(\varphi)\Big)\,\mathrm{d}x\,\mathrm{d}t\notag\\
    &\quad=\frac{1}{\tau}\int_{\tau}^{2\tau}\mathcal{E}(t)\,\mathrm{d}t-\frac{1}{\tau}\int_{T-2\tau}^{T-\tau}\int_\Omega\rho(\varphi)\frac{|\mathbf{v}|^2}{2}\,\mathrm{d}x\,\mathrm{d}t \notag\\
    &\qquad -\frac{1}{\tau}\int_{T-2\tau}^{T-\tau}\int_\Omega \Big(\frac{1}{2}|\nabla\varphi|^2+\Psi(\varphi)\Big)\,\mathrm{d}x\,\mathrm{d}t\notag\\
    &\quad=:\sum_{i=1}^3\mathcal{J}_i(\tau). \label{tau-limit-2}
\end{align}
Due to the energy inequality \eqref{energy-inequality}, it holds
$$
\mathcal{E}(t)\leq \mathcal{E}(0),\quad \forall\, t\geq 0.
$$
On the other hand, from the regularity of the weak solution $(\mathbf{v},\varphi)$ (recall Proposition \ref{weak-solution}-$(1)$), we find
$$
\mathcal{E}(0)\leq \liminf_{t\to0^+}\mathcal{E}(t).
$$
As a consequence, it holds
$$\lim_{t\to0^+}\mathcal{E}(t)= \mathcal{E}(0),$$
that is, $\mathcal{E}(t)$ is right-continuous at $t=0$. This implies that $$\lim_{\tau\to0}\mathcal{J}_1(\tau)=\mathcal{E}(0).$$
Moreover, since $\mathbf{v}=\mathbf{0}$ on $\partial\Omega\times(0,T)$, by \eqref{conditional-regularity} and
Poincar\'e's inequality, we have
$$\mathbf{v}\in L^q(0,T;\mathbf{W}^{1,r}_0(\Omega)).$$
On the other hand, condition \eqref{conditional-regularity'}, together with the embedding
$\mathbf{L}^{r_*}(\Omega) \hookrightarrow \mathbf{W}^{-1,r'}(\Omega)$
yields
$$
\partial_t\mathbf{v}
\in L^{q'}(0,T;\mathbf{W}^{-1,r'}(\Omega)).
$$
Therefore, we infer from \cite[Proposition 23.23]{Zeidler} that
$$
\mathbf{v}\in C([0,T];\mathbf{L}^2(\Omega)).
$$
Furthermore, for every $s\in(0,T)$,  Proposition \ref{weak-solution}-$(5)$ and the Aubin--Lions--Simon lemma imply that $\varphi\in C([s,T];W^{1,6}(\Omega))\hookrightarrow C([s,T];C(\overline{\Omega}))$. Then, we can pass to the limit as $\tau\to0$ in $\mathcal{J}_2$,  $\mathcal{J}_3$, and obtain
\begin{align*}
&\lim_{\tau\to0}\mathcal{J}_2(\tau)=-\int_\Omega \rho(\varphi(T))\frac{|\mathbf{v}(T)|^2}{2}\,\mathrm{d}x,
\\
&\lim_{\tau\to0}\mathcal{J}_3(\tau)=-\int_\Omega\Big(\frac{1}{2}|\nabla\varphi(T)|^2+\Psi(\varphi(T))\Big)\,\mathrm{d}x.
\end{align*}
Collecting \eqref{delta-limit-form}, \eqref{tau-limit-1}, \eqref{tau-limit-2} and the above limits for $\mathcal{J}_i$, $i=1,2,3$, we can conclude that
\begin{align}
		\mathcal{E}(T)
        +2\int_0^T\int_\Omega \nu(\varphi(t))|\mathbb{D}\mathbf{v}(t)|^2\,\mathrm{d}x\,\mathrm{d}t
        +\int_0^T\int_\Omega |\nabla\mu(t)|^2\,\mathrm{d}x\,\mathrm{d}t
        = \mathcal{E}(0). \notag
        %\label{0207-en-1}
\end{align}
Since $T>0$ is arbitrary, we see that $\mathcal{E}$ is continuous on $[0,\infty)$ and the energy equality \eqref{energy} holds.

The proof of Theorem \ref{criterion} is complete.
\hfill$\square$

%%%%%%%%%%%%%%%%%%%%%%%%%%%%%%%%
\section{Lyapunov Stability}
\label{section 4}

In this section, we prove Theorem \ref{global-strong} on the existence of a unique global strong solution and Lyapunov stability of the AGG model \eqref{AGG}--\eqref{BC-IC}.
To begin with, we state a result on the existence and uniqueness of local strong solutions.
\begin{pro}
	\label{local-existence}
	Let $\Omega\subset\mathbb{R}^3$ be a smooth bounded domain
    and let the assumption $\mathbf{(H)}$ hold.
    Assume that $(\mathbf{v}_0,\varphi _0)\in \mathbf{H}_\sigma^1(\Omega)\times H^2(\Omega)$, $\partial_{\mathbf{n}}\varphi_0=0$ on $\partial\Omega$ and satisfies \eqref{initial-mu0}. Then, there exists a time $T_0>0$, depending only on $M$, $\mathcal{E}(0)$, and coefficients of the system,
     such that problem \eqref{AGG}--\eqref{BC-IC} admits (at least) a strong solution $(\mathbf{v},p,\varphi,\mu)$ on $[0,T_0]$:
	\begin{align*}
		&\mathbf{v}\in C([0,T_0];\mathbf{H}^1_\sigma(\Omega))\cap L^2(0,T_0;\mathbf{H}^2_\sigma(\Omega))
        \cap H^1(0,T_0;\mathbf{L}^2_\sigma(\Omega)),\\
		&\varphi\in L^\infty(0,T_0;W^{2,6}(\Omega)),\\
        &\partial_t\varphi\in L^\infty(0,T_0;(H^{1}(\Omega))')\cap L^2(0,T_0;H^1(\Omega)),\\
        &\varphi\in L^\infty(\Omega\times(0,T_0))\ \
        \text{such that}\ \ |\varphi(x,t)|<1\ \text{a.e. in }\Omega\times(0,T_0), \\
		&\mu\in L^\infty(0,T_0;H^1(\Omega))\cap L^2(0,T_0;H^3(\Omega)), \\
        &p\in L^2(0,T_0;H^1(\Omega)).
	\end{align*}
    In addition, if $\|\varphi _0\|_{L^\infty}\leq 1-\delta_{0}$, for some $\delta_0\in (0,1)$,
    then there exists a time $\widetilde{T}_0\in(0,T_0]$, depending further on $\delta_0$, but independent of the specific initial datum, such that the solution is unique on the time interval $[0,\widetilde{T}_0]$, satisfying
    $\varphi\in L^\infty(0,\widetilde{T}_0;H^3(\Omega))$ and the following strict separation property:
    \begin{align}
        \|\varphi(t)\|_{C(\overline{\Omega})}\leq 1-\frac{1}{2}\delta_{0},\quad\forall\, t\in[0,\widetilde{T}_0].\label{local-separation}
    \end{align}
\end{pro}
\begin{proof}
    The proof of Proposition \ref{local-existence} can be carried out rigorously through a semi-Galerkin approximation scheme (cf. \cite{Gior22,GOW}), based on the recent contribution  \cite{CGGS} on the Cahn--Hilliard equation with a non-constant gradient energy coefficient
    and a non-degenerate mobility. Below we only derive necessary \emph{a priori} estimates.
    According to \cite[(3.1')]{AGG2011}, equation \eqref{AGG}$_1$ can be reformulated as
\begin{align}
    \partial_t(\rho(\varphi)\mathbf{v})
    +\text{div}(\mathbf{v}\otimes(\rho(\varphi)\mathbf{v}+\mathbf{J}))
    -\text{div}(2\nu(\varphi)\mathbb{D}\mathbf{v})+\nabla g
    =\mu\nabla\varphi,\label{reformulate}
\end{align}
with a modified pressure
$$
g=p+\Psi(\varphi)+a(\varphi)\frac{|\nabla\varphi|^2}{2}.
$$
Testing \eqref{reformulate} by $\mathbf{v}$ and \eqref{AGG}$_3$ by $\mu$,
then adding the two resultants, an integration in time over $[0,t]$ with $0\leq t\leq T$ yields that
\begin{align*}
    \mathcal{E}(t)
    +\int_0^t\int_\Omega 2\nu(\varphi(s))|\mathbb{D}\mathbf{v}(s)|^2\,\mathrm{d}x\,\mathrm{d}s
    +\int_0^t\int_\Omega b(\varphi(s))|\nabla\mu(s)|^2\,\mathrm{d}x\,\mathrm{d}s
    =\mathcal{E}(0).
\end{align*}
Therefore, we easily obtain
\begin{align}
    \|\mathbf{v}\|_{L^\infty(0,T;\mathbf{L}_\sigma^2(\Omega))}
    +\|\mathbf{v}\|_{L^2(0,T;\mathbf{H}_\sigma^1(\Omega))}
    +\|\varphi\|_{L^\infty(0,T;{H}^1(\Omega))}
    +\|\nabla\mu\|_{L^2(0,T;\mathbf{L}^2(\Omega))}\leq C.\label{0206-uni1}
\end{align}
Define $A(\varphi):=\int_0^\varphi\sqrt{a(s)}\,\mathrm{d}s$ as in \cite{ADG2013}. We observe that
\begin{align*}
    |A(\varphi(x,t))|\leq \sqrt{a^\ast} \ \ \text{for all }(x,t)\in \Omega\times(0,T)
    \quad\text{and}\quad \nabla A(\varphi)=\sqrt{a(\varphi)}\nabla\varphi,
\end{align*}
which together with \eqref{0206-uni1} imply
$$
\|A(\varphi)\|_{L^\infty(0,T;H^1(\Omega))}\leq C.
$$
Moreover, \eqref{AGG}$_3$ and \eqref{0206-uni1} yield that $$\|\partial_t\varphi\|_{L^2(0,T;(H^{1}(\Omega))')}\leq C.$$
Following the same argument as  \cite[(5.13)--(5.27)]{CGGS}, we obtain
\begin{align*}
    &\|\varphi\|_{L^4(0,T;H^2(\Omega))} +\|A(\varphi)\|_{L^4(0,T;H^2(\Omega))}+\|\varphi\|_{L^2(0,T;W^{2,6}(\Omega))}\leq C,\\
    &\|\mu\|_{L^2(0,T;H^1(\Omega))} +\|\Psi_0'(\varphi)\|_{L^2(0,T;L^6(\Omega))}\leq C.
\end{align*}
Next, testing \eqref{reformulate} by $\partial_t\mathbf{v}$, we obtain
\begin{align}
   & \frac{\mathrm{d}}{\mathrm{d}t}\int_\Omega \nu(\varphi)|\mathbb{D}\mathbf{v}|^2\,\mathrm{d}x
   +\int_\Omega \rho(\varphi)|\partial_t\mathbf{v}|^2\,\mathrm{d}x\notag\\
   &\quad=\int_\Omega \nu'(\varphi)\partial_t\varphi|\mathbb{D}\mathbf{v}|^2\,\mathrm{d}x
   -\int_\Omega \rho(\varphi)[(\mathbf{v}\cdot\nabla)\mathbf{v}]\cdot\partial_t\mathbf{v}\,\mathrm{d}x\notag\\
    &\qquad+\int_\Omega \mu\nabla\varphi\cdot\partial_t\mathbf{v}\,\mathrm{d}x
    +\int_\Omega \rho'(\varphi) b(\varphi) [(\nabla\mu \cdot\nabla)\mathbf{v}]\cdot\partial_t\mathbf{v}\,\mathrm{d}x.\label{1210-1}
\end{align}
Testing \eqref{AGG}$_3$ by $\partial_t\mu$, we get (cf. \cite[(9.11)]{CGGS})
\begin{align}
    &\frac{\mathrm{d}}{\mathrm{d}t}\left(\frac{1}{2}\int_\Omega b(\varphi)|\nabla\mu|^2\,\mathrm{d}x
    +\int_\Omega\mu \nabla\varphi\cdot\mathbf{v}\,\mathrm{d}x\right)
    +\int_\Omega a(\varphi)|\nabla\partial_t\varphi|^2\,\mathrm{d}x
    \notag \\
    &\qquad +\int_\Omega \frac{(a'(\varphi))^2}{4a(\varphi)}|\nabla\varphi|^2|\partial_t\varphi|^2\,\mathrm{d}x\notag\\
    &\quad=\int_\Omega \mu\nabla\varphi\cdot\partial_t\mathbf{v}\,\mathrm{d}x
    +\int_\Omega \mu\nabla\partial_t\varphi\cdot\mathbf{v}\,\mathrm{d}x
    -\int_\Omega \Psi''(\varphi)|\partial_t\varphi|^2\,\mathrm{d}x
    \notag \\
    &\qquad
    +\int_\Omega\frac{a'(\varphi)}{2\sqrt{a(\varphi)}}\Delta A(\varphi)|\partial_t\varphi|^2\,\mathrm{d}x
    -\int_\Omega a'(\varphi)\nabla\varphi\cdot\nabla\partial_t\varphi\, \partial_t\varphi\,\mathrm{d}x
    \notag\\
    &\qquad
    +\frac{1}{2}\int_\Omega b'(\varphi)\partial_t\varphi|\nabla\mu|^2\,\mathrm{d}x.
    \label{1210-2}
\end{align}
Define
	\begin{align}
		\mathcal{H}(t):=\int_\Omega \nu(\varphi(t))|\mathbb{D}\mathbf{v}(t)|^2\,\mathrm{d}x
		+\frac{1}{2}\int_\Omega b(\varphi(t))|\nabla\mu(t)|^2\,\mathrm{d}x
		+\int_\Omega\mu(t)\mathbf{v}(t)\cdot\nabla\varphi(t)\,\mathrm{d}x.\notag
	\end{align}
We deduce from \eqref{1210-1} and \eqref{1210-2} that
\begin{align}
&\frac{\mathrm{d}}{\mathrm{d}t}\mathcal{H}+\int_\Omega \rho(\varphi)|\partial_t\mathbf{v}|^2\,\mathrm{d}x
+\int_\Omega a(\varphi)|\nabla\partial_t\varphi|^2\,\mathrm{d}x\notag\\
&\quad\leq\int_\Omega \nu'(\varphi)\partial_t\varphi|\mathbb{D}\mathbf{v}|^2\,\mathrm{d}x
-\int_\Omega \rho(\varphi)[(\mathbf{v}\cdot\nabla)\mathbf{v}]\cdot\partial_t\mathbf{v}\,\mathrm{d}x
+2\int_\Omega \mu\nabla\varphi\cdot\partial_t\mathbf{v}\,\mathrm{d}x\notag\\
&\qquad+\int_\Omega \rho'(\varphi) b(\varphi) [(\nabla\mu \cdot\nabla)\mathbf{v}]\cdot\partial_t\mathbf{v}\,\mathrm{d}x
+\int_\Omega \mu\nabla\partial_t\varphi\cdot\mathbf{v}\,\mathrm{d}x
\notag \\
&\qquad +\Theta_0\int_\Omega |\partial_t\varphi|^2\,\mathrm{d}x
+\frac{1}{2}\int_\Omega b'(\varphi)\partial_t\varphi|\nabla\mu|^2\,\mathrm{d}x
\notag\\
&\qquad
+\int_\Omega\frac{a'(\varphi)}{2\sqrt{a(\varphi)}}\Delta A(\varphi)|\partial_t\varphi|^2\,\mathrm{d}x-\int_\Omega a'(\varphi)\nabla\varphi\cdot\nabla\partial_t\varphi\, \partial_t\varphi\,\mathrm{d}x\notag\\
&\quad=:\sum_{j=1}^{9}\mathcal{K}_j.\label{0206-H}
\end{align}
Using \eqref{0206-uni1} and following a similar argument as the derivation of \cite[(3.53)]{Gior22}, we obtain
\[\mathcal{H}\geq \frac{1}{2}\int_\Omega \nu(\varphi)|\mathbb{D}\mathbf{v}|^2\,\mathrm{d}x+\frac{1}{4}\int_\Omega b(\varphi)|\nabla\mu|^2\,\mathrm{d}x-C_0,\]
for some positive constant $C_0$ depending on $\|\mathbf{v}_0\|_{L^2}$, $\|\mathbf{\varphi}_0\|_{H^1}$, $\overline{\varphi_0}$, $\nu_\ast$, $b_\ast$ and $\Omega$. For the terms $\mathcal{K}_j$ with $1\leq j\leq6$, we use the estimates in \cite{Gior22} and conclude that
\begin{align}
    \sum_{j=1}^6|\mathcal{K}_j|
    \leq \tilde{\epsilon}\Big(\|\mathbf{A}\mathbf{v}\|_{L^2}^2+\|\partial_t\mathbf{v}\|_{L^2}^2+\|\nabla\partial_t\varphi\|_{L^2}^2+\|\nabla\mu\|_{H^2}^2\Big)
    +C_{ \tilde{\epsilon}}(1+C_0+\mathcal{H})^5,\label{1-6}
\end{align}
for a small constant $ \tilde{\epsilon}\in (0,1)$.
Concerning $\mathcal{K}_7$, $\mathcal{K}_8$ and $\mathcal{K}_9$, the estimates in \cite[(9.11)]{CGGS} imply that
\begin{align}
    |\mathcal{K}_7|+|\mathcal{K}_8|+|\mathcal{K}_9|\leq \tilde{\epsilon}\|\nabla\partial_t\varphi\|_{L^2}^2+C_{\tilde{\epsilon}}(1+C_0+\mathcal{H})^3.\label{78}
\end{align}
From \eqref{0206-H}, \eqref{1-6} and \eqref{78}, it follows  that
\begin{align}
    &\frac{\mathrm{d}}{\mathrm{d}t}\mathcal{H}+\int_\Omega \rho(\varphi)|\partial_t\mathbf{v}|^2\,\mathrm{d}x
    +a_\ast\int_\Omega|\nabla\partial_t \varphi|^2\,\mathrm{d}x\notag\\
    &\quad\leq \tilde{\epsilon}\Big(\|\mathbf{A}\mathbf{v}\|_{L^2}^2+\|\partial_t\mathbf{v}\|_{L^2}^2+\|\nabla\partial_t\varphi\|_{L^2}^2+\|\nabla\mu\|_{H^2}^2\Big)
    +C_{\tilde{\epsilon}}(1+C_0+\mathcal{H})^5.\label{1210-3}
\end{align}
We now control the higher-order terms $\|\nabla\mu\|_{H^2}$ and $\|\mathbf{A}\mathbf{v}\|_{L^2}$ on the right-hand side of \eqref{1210-3}. First, we take the gradient in \eqref{AGG}$_3$ and get
\begin{align*}
    b(\varphi)\nabla\Delta\mu
    &=\nabla\partial_t\varphi +\nabla\mathbf{v}\nabla\varphi+\nabla^2\varphi\mathbf{v}-b'(\varphi)\Delta\mu\nabla\varphi\\
    &\quad-b'(\varphi)\nabla^2\mu\nabla\varphi-b'(\varphi)\nabla^2\varphi\nabla\mu-b''(\varphi)(\nabla\varphi\cdot\nabla\mu)\nabla\varphi.
\end{align*}
Then, an argument similar to that in \cite{GOW} yields that
    \begin{align}
        \|\nabla\mu\|_{H^2}^2\leq \tilde{\epsilon}(\|\nabla\mu\|_{H^2}^2
        +\|\mathbf{A}\mathbf{v}\|_{L^2}^2)+C_{\tilde{\epsilon}}(1+C_0+\mathcal{H})^5
        +C_{\tilde{\epsilon}}\|\nabla\partial_t\varphi\|_{L^2}^2.\label{1210-4}
    \end{align}
    For the term $\|\mathbf{A\mathbf{v}}\|_{L^2}$,
    according to \cite[(3.57)]{Gior22} and the corresponding estimates for $R_i$ therein with $i=7,...,12$, we can obtain
\begin{align}
    \|\mathbf{Av}\|_{L^2}^2\leq \tilde{\epsilon}(\|\mathbf{Av}\|_{L^2}^2+\|\nabla\mu\|_{H^2}^2)
    +C_{\tilde{\epsilon}}\|\partial_t\mathbf{v}\|_{L^2}^2 +C_{\tilde{\epsilon}}(1+C_0+\mathcal{H})^5.\label{1210-5}
\end{align}
Collecting \eqref{1210-3}, \eqref{1210-4} and \eqref{1210-5}, choosing sufficiently small $\tilde{\epsilon}>0$, we arrive at
\begin{align}
    \frac{\mathrm{d}}{\mathrm{d}t}\mathcal{H}+\frac{\rho_\ast}{2}\|\partial_t\mathbf{v}\|_{L^2}^2+\frac{a_\ast}{2}\|\nabla\partial_t\varphi\|_{L^2}^2
    +\widetilde{C}(\|\mathbf{A}\mathbf{v}\|_{L^2}^2+\|\mu\|_{H^3}^2)
    \leq C(1+C_0+\mathcal{H})^5.\label{0206-1}
\end{align}
By \eqref{initial-mu0}, \eqref{0206-1} and Korn's inequality \eqref{Korn}, we can conclude that there exists a positive time $T_0$, depending on $\|\mathbf{v}_0\|_{H^1}$, $\|\mu_0\|_{H^1}$, $\mathcal{E}(0)$ and coefficients of the system, such that
\begin{align}
    \|\nabla\mathbf{v}\|_{L^\infty(0,T_0; \mathbf{L}_\sigma^2(\Omega))}
    +\|\nabla\mu\|_{L^\infty(0,T_0;\mathbf{L}^2(\Omega))}\leq C,
    \label{0206-i2}
\end{align}
which, together with \eqref{AGG}$_3$ yields
$$
\|\partial_t\varphi\|_{L^\infty(0,T_0;(H^{1}(\Omega))')}\leq C.
$$
According to the estimates \cite[(5.17), (5.23) and (5.26)]{CGGS}, from \eqref{0206-i2}, we infer that
\begin{align}
    \|\mu\|_{L^\infty(0,T_0;H^1(\Omega))} +\|\varphi\|_{L^\infty(0,T_0;W^{2,6}(\Omega))}+\|\Psi_0'(\varphi)\|_{L^\infty(0,T_0;L^6(\Omega))}\leq C.\notag
\end{align}
Integrating \eqref{0206-1} on $[0,T_0]$, we further get
\begin{align*}
    \|\partial_t\mathbf{v}\|_{L^2(0,T_0; \mathbf{L}^2(\Omega))} +\|\nabla\partial_t\varphi\|_{L^2(0,T_0; \mathbf{L}^2(\Omega))}
    +\|\mathbf{Av}\|_{L^2(0,T_0; \mathbf{L}^2(\Omega))} +\|\mu\|_{L^2(0,T_0;H^3(\Omega))}\leq C.
\end{align*}
When the initial datum $\varphi_0$ satisfies $\|\varphi_0\|_{L^\infty}\leq 1-\delta_0$,
following the same argument as in \cite[Section 9]{CGGS},
we can conclude that there exists a smaller time $\widetilde{T}_0\in(0,T_0]$ such that $\varphi\in L^\infty(0,\widetilde{T}_0;H^3(\Omega))$ and \eqref{local-separation} holds. For the uniqueness of strong solutions on $[0,\widetilde{T}_0]$, we refer to \cite{GOW,CGGS} and omit the details here. This completes the proof of Proposition \ref{local-existence}.
\end{proof}

The following energy identity is crucial for establishing the existence of global strong solutions. The proof follows from direct calculations, which are guaranteed by the regularity of the strong solution.
\begin{pro}
	Let the assumptions in Proposition \ref{local-existence} be satisfied.
    Then, the local strong solution $(\mathbf{v},p,\varphi,\mu)$ to problem \eqref{AGG}--\eqref{BC-IC} satisfies the following energy identity:
	\begin{align}
		\frac{\mathrm{d}}{\mathrm{d}t}\mathcal{E}(t)
        +2\Big\|\sqrt{\nu(\varphi(t))}\mathbb{D}\mathbf{v}(t)\Big\|_{L^2}^2
        +\Big\|\sqrt{b(\varphi(t))} \nabla\mu(t)\Big\|_{L^2}^2
        =0,\label{en-id}
	\end{align}
	for almost all $t\in(0,T_0)$.
\end{pro}

\subsection{\texorpdfstring{Bounds for initial data near  $(\mathbf{0},\varphi_\ast)$}{Bounds for initial data near (\mathbf{0},\varphi_\ast)}}
Since $\varphi _\ast$ is a local minimizer of $E_{\mathrm{free}}$ on $\mathcal{V}_m$,
\cite[Proposition 3.2]{CGGS} indicates that $\varphi_\ast\in H^3(\Omega)$ and there exist $\xi\in(0,1)$ and $\chi>0$ such that
\begin{equation}
\begin{aligned}
	& \|\varphi _\ast\|_{C(\overline{\Omega})}\leq 1-\xi, \\
    & E_{\text{free}}(\varphi_\ast)\leq E_{\text{free}}(\varphi)\
    \text{ for all }\varphi\in\mathcal{V}_m: \ \|\varphi_\ast-\varphi\|_{H^1}\leq\chi.
\end{aligned}
\label{p1}
\end{equation}
For convenience, let us recall the Sobolev embedding theorem in three dimensions, the Poincar\'e--Wirtinger inequality and Poincar\'e's inequality:
\begin{align}
	& \|f\|_{C(\overline{\Omega})}\leq C_\mathrm{S}\|f\|_{H^2},\quad\qquad\ \ \ \ \ \forall\, f\in H^2(\Omega),\label{SoB}\\
    & \|f\|_{H^1}\leq C_\mathrm{W}(\|\nabla f\|_{L^2}+|\overline{f}|),\quad\forall\, f\in H^1(\Omega),\label{PW}\\
    & \|f\|_{H^1}\leq C_\text{P}\|\nabla f\|_{L^2},\quad\qquad\quad\ \ \forall\, f\in H_0^1(\Omega),\label{Po}
\end{align}
where the positive constants $C_\mathrm{S}$, $C_{\mathrm{W}}$ and $C_\mathrm{P}$ depend only on $\Omega$.
According to the assumptions in Theorem \ref{global-strong}, we assume
\begin{align}
	&\|\mathbf{v}_0\|_{H^1}\leq M,\qquad \|\mu_0\|_{H^1}\leq M,\label{b1}\\
	&\|\mathbf{v}_0\|_{L^2}\leq \eta_1,\qquad\|\varphi_0-\varphi _\ast\|_{H^2}\leq \eta_2,\label{b2}
\end{align}
where $\eta_1,$ $\eta_2\in(0,1)$, $M>0$ are constants to be determined later.
In \eqref{b1}, we have set
\[\mu_0=-\text{div}(a(\varphi_0)\nabla\varphi_0)+a'(\varphi_0)\frac{|\nabla\varphi_0|^2}{2}+\Psi'(\varphi_0).\]
Due to \eqref{b2} and the fact $\eta_2<1$, we find
\begin{align}
	\|\varphi _0\|_{H^1}\leq \|\varphi_\ast\|_{H^1}+1,
    \quad \|\varphi  _0\|_{H^2}\leq \|\varphi _\ast\|_{H^2}+1.\label{b3}
\end{align}
Moreover, taking
\[0<\eta_2<\min\Big\{1,\frac{\xi}{3C_\mathrm{S}}\Big\},\]
from \eqref{p1}, \eqref{SoB} and \eqref{b2}, we infer that
\begin{align}
	\|\varphi _0\|_{C(\overline{\Omega})}
    & \leq \|\varphi _\ast\|_{C(\overline{\Omega})}+\|\varphi_0-\varphi _\ast\|_{C(\overline{\Omega})}
    \notag \\
    & \leq \|\varphi _\ast\|_{C(\overline{\Omega})}+C_\mathrm{S}\|\varphi_0-\varphi _\ast\|_{H^2}
    \notag \\
    & \leq 1-\xi+C_\mathrm{S}\eta_2\leq  1-\frac{2\xi}{3}.
    \label{s1}
\end{align}
Let
\begin{align}
	K_1=\max_{s\in[-1+\frac{\xi}{3},1-\frac{\xi}{3}]}|\Psi''(s)|,
    \quad K_2=\max_{s\in[-1+\frac{\xi}{3},1-\frac{\xi}{3}]}|\Psi'(s)|,
    \quad K_3=\max_{s\in[-1,1]}|\Psi(s)|.\label{n1}
\end{align}
Then by \eqref{b3} and \eqref{n1}, we have
\begin{align}
	E_{\text{free}}(\varphi _0)
    &\leq \frac{a^\ast}{2}\|\nabla\varphi _0\|_{L^2}^2+|\Omega|\max_{s\in[-1,1]}|\Psi(s)|
    \notag \\
    & \leq \frac{a^\ast}{2}(\|\varphi _\ast\|_{H^2}+1)^2+|\Omega|K_3=:L_1.
    \label{b4}
\end{align}

Thanks to Proposition \ref{local-existence}, for the initial data $(\mathbf{v}_0,\varphi _0)$ satisfying the assumptions in Theorem \ref{global-strong},
there exists a universal time $T_1>0$, depending only on $M$, $L_1$ and coefficients of the system, such that problem \eqref{AGG}--\eqref{BC-IC} admits a unique local strong solution $(\mathbf{v},p,\varphi,\mu)$ on $[0,T_1]$, satisfying
\begin{align}
	\|\varphi(t)\|_{C(\overline{\Omega})}\leq 1-\frac{\xi}{3},\quad\forall \,t\in[0,T_1],\label{s2}
\end{align}
and
\begin{align}
    \sup_{t\in[0,T_1]}\|\varphi(t)\|_{H^3}:=M_0,\label{M0}
\end{align}
for some $M_0>0$ depending on $M$, $L_1$ and $m$, but independent of the specific data $\varphi_0$.
Based on the energy identity \eqref{en-id}, the estimates \eqref{b2}, \eqref{b4}, we obtain
\begin{align}
	\mathcal{E}(t)
    &=\frac{1}{2}\int_\Omega{\rho(\varphi(t))}|\mathbf{v}(t)|^2\,\mathrm{d}x
    +\frac{1}{2}\int_\Omega a(\varphi(t))|\nabla\varphi(t)|^2\,\mathrm{d}x
    +\int_\Omega\Psi(\varphi(t))\,\mathrm{d}x\notag\\
	&\leq \mathcal{E}(0)
    =\frac{1}{2}\Big\|\sqrt{\rho(\varphi_0)}\mathbf{v}_0\Big\|_{L^2}^2
    +E_{\text{free}}(\varphi _0)
    \notag \\
    &\leq \frac{\rho^\ast}{2}+L_1,\quad\forall\, t\in[0,T_1],\notag%\label{b7}
\end{align}
which, together with the Poincar\'e--Wirtinger inequality \eqref{PW}, implies that
\begin{equation}
\begin{aligned}
	\|\mathbf{v}(t)\|_{L^2}^2&\leq  \frac{\rho^\ast+2L_1+2|\Omega|K_3}{\rho_\ast},\\
     \|\varphi(t)\|_{H^1}^2&\leq C_{\mathrm{W}}^2(\|\nabla\varphi(t)\|_{L^2}+|{m}|)^2\\
     &\leq 2C_{\mathrm{W}}^2(\|\nabla\varphi(t)\|_{L^2}^2+{m}^2)\\
    &\leq 2C_\mathrm{W}^2\Big(\frac{\rho^\ast+2L_1+2|\Omega|K_3}{a_\ast}+m^2\Big)=:L_2^2,
\end{aligned}
\label{b10}
\end{equation}
for all $t\in[0,T_1]$.
Besides, according to \cite[(5.17)]{CGGS}, there exists a constant $L_3>0$, depending only on $m$, $L_1$ and $\Omega$, such that
\begin{align}
    \|\mu(t)\|_{H^1}\leq L_3(\|\nabla\mu(t)\|_{L^2}+1).\label{0205-mu}
\end{align}

We proceed to derive higher-order estimates. Consider the function
\begin{align}
	\widetilde{\mathcal{H}}(t)
    & =\int_\Omega \nu(\varphi(t))|\mathbb{D}\mathbf{v}(t)|^2\,\mathrm{d}x
    +\frac{\beta_\ast}{2}\int_\Omega b(\varphi(t))|\nabla\mu(t)|^2\,\mathrm{d}x
    \notag \\
    &\quad +\beta_\ast\int_\Omega\mu(t)\mathbf{v}(t)\cdot\nabla\varphi(t)\,\mathrm{d}x,\label{n2}
\end{align}
where the constant $\beta_\ast$ is chosen as
\[\beta_\ast:=\min\left\{\frac{\nu_\ast b_\ast}{8C_\mathrm{P}^2},\ \frac{1}{2}\right\}\in(0,1).\]
Using H\"older's inequality, Poincar\'e's inequality \eqref{Po}, and the interpolation inequality, we find
\begin{align*}
	\Big|\beta_\ast\int_\Omega\mu(t)\mathbf{v}(t)\cdot\nabla\varphi(t)\,\mathrm{d}x\Big|
    &=\Big|\beta_\ast\int_\Omega\varphi(t)\mathbf{v}(t)\cdot\nabla\mu(t)\,\mathrm{d}x\Big|\\
	&\leq \beta_\ast\|\varphi(t)\|_{L^\infty}\|\mathbf{v}(t)\|_{L^2}\|\nabla\mu(t)\|_{L^2}\\
	&\leq \beta_\ast {C_\mathrm{P}}\|\nabla\mathbf{v}(t)\|_{L^2}\|\nabla\mu(t)\|_{L^2}\\
	&\leq \frac{\beta_\ast C_\mathrm{P}^2}{b_\ast}\|\nabla\mathbf{v}(t)\|_{L^2}^2+\frac{\beta_\ast b_\ast}{4}\|\nabla\mu(t)\|_{L^2}^2\\
	&\leq \frac{\nu_\ast }{4}\|\nabla\mathbf{v}(t)\|_{L^2}^2+\frac{\beta_\ast b_\ast}{4}\|\nabla\mu(t)\|_{L^2}^2,
\end{align*}
which, together with the definition of $\widetilde{\mathcal{H}}(t)$, implies that
\begin{align}
	\widetilde{\mathcal{H}}(t)
    \geq\frac{1}{4}\int_\Omega \nu(\varphi
    (t))|\mathbb{D}\mathbf{v}(t)|^2\,\mathrm{d}x
    +\frac{\beta_\ast }{4}\int_\Omega b(\varphi(t))|\nabla\mu(t)|^2\,\mathrm{d}x,\quad\forall\,t\in[0,T_1].\label{b12}
\end{align}
By \eqref{b1}, \eqref{b2}, \eqref{s1}, Korn's inequality \eqref{Korn} and integration-by-parts, we obtain
\begin{align}
	\widetilde{\mathcal{H}}(0)
    &=\|\sqrt{\nu(\varphi_0)}\mathbb{D}\mathbf{v}_0\|_{L^2}^2
    +\frac{\beta_\ast}{2}\|\sqrt{b(\varphi_0)}\nabla\mu(0)\|_{L^2}^2
    +\beta_\ast\int_\Omega\mu(0)\mathbf{v}_0\cdot\nabla\varphi _0\,\mathrm{d}x\notag\\
	&=\|\sqrt{\nu(\varphi _0)}\mathbb{D}\mathbf{v}_0\|_{L^2}^2
    +\frac{\beta_\ast}{2}\|\sqrt{b(\varphi_0)}\nabla\mu(0)\|_{L^2}^2
    -\beta_\ast\int_\Omega\varphi _0\nabla\mu(0)\cdot\mathbf{v}_0\,\mathrm{d}x\notag\\
	&\leq \nu^\ast\|\nabla\mathbf{v}_0\|_{L^2}^2
    +\frac{\beta_\ast b^\ast}{2}\|\nabla\mu(0)\|_{L^2}^2
    +\beta_\ast\|\varphi _0\|_{L^\infty}\|\nabla\mu(0)\|_{L^2}\|\mathbf{v}_0\|_{L^2}\notag\\
	&\leq \nu^\ast\|\nabla\mathbf{v}_0\|_{L^2}^2
    +\frac{\beta_\ast b^\ast}{2}\|\nabla\mu(0)\|_{L^2}^2
    +{C_\mathrm{P}} \beta_\ast \|\nabla\mu(0)\|_{L^2}\|\nabla\mathbf{v}_0\|_{L^2}\notag\\
	&\leq \Big(\nu^\ast+\frac{\beta_\ast b^\ast}{2}+C_\text{P}\beta_\ast\Big)M^2=:M_1.
\label{b5}
\end{align}
Similarly to \eqref{0206-1}, we can derive the following differential inequality:
\begin{align}
	\frac{\mathrm{d}}{\mathrm{d}t}\widetilde{\mathcal{H}}(t)\leq C[1+(\widetilde{\mathcal{H}}(t))^5],\quad\forall\,t\in(0,T_1),
    \label{D1}
\end{align}
where the positive constant $C$ depends on $L_{1}$, $\xi$, $\nu_\ast$, $\nu^\ast$, $b_\ast$, $b^\ast$, and $\Omega$.
%%%%
Owing to \eqref{b5} and \eqref{D1}, there exists $T_2\in(0,T_1]$ such that
\begin{align}
	\widetilde{\mathcal{H}}(t)\leq 2M_1,\quad\forall\,t\in[0,T_2],\notag%\label{b6}
\end{align}
where the time $T_2$ depends on $M_1$, $M$, $L_1$, $\nu_\ast$, $\nu^\ast$, $b_\ast$, $b^\ast$, and $\Omega$, but is independent of the specific data  $(\mathbf{v}_0,\varphi _0)$.
Now, we choose a sufficiently large constant $M$, such that
\begin{align}
M\geq\max\left\{\Big(\frac{2L_3}{\sqrt{\beta_\ast b_\ast}}+L_3\Big),\ \frac{8C_\mathrm{P}}{\sqrt{\nu_\ast}}\right\}.\label{M}
\end{align}
With the choice of $M$, the constant $M_1$ is also determined (cf. \eqref{b5}). As a consequence, the time $T_2$ can be determined.
Recalling \eqref{b12}, we can obtain the following higher-order estimate
\begin{align}
	\|\mathbf{v}(t)\|_{H^1}
    \leq \frac{8C_\mathrm{P}}{\sqrt{\nu_\ast}}\sqrt{\widetilde{\mathcal{H}}(t)}
    \leq \frac{8C_\mathrm{P}}{\sqrt{\nu_\ast}}\sqrt{2M_1},\quad\forall\,t\in[0,T_2].\label{b13}
\end{align}

\subsection{Refined estimates}
\label{refine-es}
In what follows, we derive some refined estimates by means of the {\L}ojasiewicz--Simon approach.
According to the boundedness of $b$ and $\nu$, together with the energy identity \eqref{en-id}, we obtain
\begin{align}
	-\frac{\mathrm{d}}{\mathrm{d}t}\mathcal{E}(t)
    &=2\Big\|\sqrt{\nu(\varphi(t))}\mathbb{D}\mathbf{v}(t)\Big\|_{L^2}^2+\Big\|\sqrt{b(\varphi(t))}\nabla\mu(t)\Big\|_{L^2}^2
    \notag \\
    & \geq 2\nu_\ast\|\mathbb{D}\mathbf{v}(t)\|_{L^2}^2+b_\ast\|\nabla\mu(t)\|_{L^2}^2.
    \label{b15}
\end{align}
By \eqref{n1}, \eqref{s2} and \eqref{b10}, and arguing as in \cite[(10.18)]{CGGS}, we obtain
\begin{align}
	&|E_{\text{free}}(\varphi _0)-E_{\text{free}}(\varphi(t))|
	\leq M_2\|\varphi(t)-\varphi _0\|_{H^2},\quad\forall\, t\in [0,T_2],\label{b16}
\end{align}
where the positive constant $M_2$ depends on $L_2$, $K_2$, and $\Omega$.

For any $\epsilon>0$, let
\begin{align}
	\omega=\min\left\{1,\,\epsilon,\,\chi,\,\kappa,\,\frac{\xi}{3C_\mathrm{S}},\,\frac{4C_\ast^{-1} E_0}{5M_2},\,\sqrt{C_\ast^{-1} E_0}\right\}\label{omega}
\end{align}
with
\[C_\ast:=\frac{\nu^\ast}{\nu_\ast} + \frac{b^\ast}{b_\ast},\quad   E_0=\frac{\min\{1,M_1\}T_2}{4}.\]
Then for $\eta_2\in(0,\omega/2]$, we set
\[T_{\eta_2}=\inf\big\{t>0:\ \|\varphi(t)-\varphi _\ast\|_{H^2}\geq\omega\big\}.\]
By the continuity of the strong solution, we have
$$T_{\eta_2}>0.$$
Next, we claim that there exists at least one $\eta_2$ such that
$$T_{\eta_2}\geq T_2.$$
If this is not true, for any $\eta_2\in(0,\omega/2]$, we have $T_{\eta_2}\leq T_2$.
First, we consider a special case, that is, there exists a time $t_0\in[0,T_{\eta_2}]$ such that $\mathcal{E}(t_0)=E_{\text{free}}(\varphi _\ast)$.
By the energy identity \eqref{en-id}, it holds
\[\|\nabla\mathbf{v}(t)\|_{L^2}=\|\nabla\mu(t)\|_{L^2}=0,\quad\forall \, t\geq t_0,\]
which implies that the solution on $[t_0,\infty)$ is independent of time and thus the solution becomes global.
Except for the above special case, we have
\[
\mathcal{E}(t)-E_{\text{free}}(\varphi _\ast)>0,\quad\text{for all }t\in[0,T_{\eta_2}].
\]
By Lemma \ref{LS}, \eqref{b13}, \eqref{b15}, and the Poincar\'e--Wirtinger inequality \eqref{PW}, we can deduce that
\begin{align}
	&-\frac{\mathrm{d}}{\mathrm{d}t}[\mathcal{E}(t)-E_{\text{free}}(\varphi _\ast)]^\theta
    \notag \\
	&\quad =-\theta[\mathcal{E}(t)-E_{\text{free}}(\varphi _\ast)]^{\theta-1}
    \frac{\mathrm{d}}{\mathrm{d}t}\mathcal{E}(t)\notag\\
    &\quad\geq \frac{ \theta (2\nu_\ast\|\mathbb{D}\mathbf{v}(t)\|_{L^2}^2+b_\ast\|\nabla\mu(t)\|_{L^2}^2)}{(\rho^\ast/2)^{1-\theta}\|\mathbf{v}(t)\|_{L^2}^{2(1-\theta)}+(E_{\mathrm{free}}(\varphi(t))-E_{\mathrm{free}}(\varphi_\ast))^{1-\theta}}\notag\\
	&\quad \geq\frac{ \theta (2\nu_\ast\|\mathbb{D}\mathbf{v}(t)\|_{L^2}^2+b_\ast\|\nabla\mu(t)\|_{L^2}^2)}
    {(\rho^\ast C_{\mathrm{P}}/\sqrt{2})^{1-\theta}\|\mathbb{D}\mathbf{v}(t)\|_{L^2}^{2(1-\theta)}+C_L\|\mu(t)-\overline{\mu(t)}\|_{L^2}}\notag\\
	&\quad \geq\frac{ \theta  (2\nu_\ast\|\mathbb{D}\mathbf{v}(t)\|_{L^2}+b_\ast\|\nabla\mu(t)\|_{L^2})^2}
    {(\rho^\ast C_{\mathrm{P}}/\sqrt{2})^{1-\theta}\big(8C_\mathrm{P}\nu_\ast^{-\frac{1}{2}}\sqrt{2M_1}\big)^{1-2\theta}\|\mathbb{D}\mathbf{v}(t)\|_{L^2}+C_L C_\mathrm{W}\|\nabla\mu(t)\|_{L^2}}\notag\\[1mm]
	&\quad \geq C_1(\|\nabla\mathbf{v}(t)\|_{L^2}+\|\nabla\mu(t)\|_{L^2})
           \geq C_1\|\partial_t\varphi(t)\|_{(H^1(\Omega))'},
    \label{b17}
\end{align}
where the positive constant $C_1$ depends on $\theta$, $\rho^\ast$, $\nu_\ast$, $b_\ast$, $C_\mathrm{P}$, $C_{\mathrm{W}}$, $M_1$, $C_L$, and $\Omega$.
Similar to the derivation of \eqref{b16} (see also \cite[(10.18)]{CGGS}), it holds
\[|E_{\text{free}}(\varphi _0)-E_{\text{free}}(\varphi _\ast)|\leq M_2\|\varphi_0-\varphi_\ast\|_{H^2},\]
which, together with \eqref{b17}, implies that
\begin{align}
	\int_0^{T_{\eta_2}}\|\partial_t \varphi(t)\|_{(H^1(\Omega))'}\,\mathrm{d}t
    &\leq \frac{1}{C_1}|\mathcal{E}(0)-E_{\text{free}}(\varphi _\ast)|^\theta\notag\\
	&\leq \frac{1}{C_1}\Big(\frac{\rho^\ast}{2}\|\mathbf{v}_0\|_{L^2}^2+M_2\|\varphi _0-\varphi _\ast\|_{H^2}\Big)^\theta\notag\\
	&\leq C_2\Big(\|\mathbf{v}_0\|_{L^2}^{2\theta}+\|\varphi _0- \varphi_\ast\|_{H^2}^\theta\Big),
    \label{b18}
\end{align}
where the positive constant $C_2$ depends on $\theta$, $\rho^\ast$, $C_1$ and $M_2$.
Combining \eqref{M0} and \eqref{b18}, we find
\begin{align}
	&\|\varphi(T_{\eta_2})-\varphi _\ast\|_{H^2}\notag \\
    &\quad \leq \|\varphi _0-\varphi _\ast\|_{H^2}+\|\varphi(T_{\eta_2})-\varphi _0\|_{H^2}\notag\\
	&\quad \leq \|\varphi _0-\varphi _\ast\|_{H^2}
    +C_3\|\varphi(T_{\eta_2})-\varphi _0\|_{H^3}^{\frac{3}{4}}\|\varphi(T_{\eta_2})-\varphi _0\|_{(H^1(\Omega))'}^{\frac{1}{4}}\notag\\
	&\quad \leq \|\varphi _0-\varphi _\ast\|_{H^2}
    +C_3 C_2^{\frac{1}{4}} (2M_0)^{\frac{3}{4}}\Big(\|\mathbf{v}_0\|_{L^2}^{\frac{\theta}{2}}+\|\varphi _0-\varphi _\ast\|_{H^2}^{\frac{\theta}{4}}\Big),\label{b19}
\end{align}
where the positive constant $C_3$ depends only on $\Omega$.
Define
\begin{equation}
\begin{aligned}
	&\eta_1:=\min\left\{\frac{\omega}{4\sqrt{\rho^\ast}},\ \left(\frac{\omega}{4C_3(2M_0)^{\frac{3}{4}} C_2^{\frac{1}{4}}}\right)^{\frac{2}{\theta}}\right\},
    \\
&\eta_2:=\min\left\{\frac{\omega}{4},\ \left(\frac{\omega}{4C_3(2M_0)^{\frac{3}{4}} C_2^{\frac{1}{4}}}\right)^{\frac{4}{\theta}}\right\}.
\end{aligned}
\label{eta2}
\end{equation}
Then, by \eqref{b19} and \eqref{eta2}, we have
\[\|\varphi(T_{\eta_2})-\varphi _\ast\|_{H^2}\leq \frac{3}{4}\omega<\omega,\]
which contradicts the definition of $T_{\eta_2}$. Hence, there exists $\eta_2\in (0,\omega/2]$ such that $T_{\eta_2}\geq T_2$.

Thanks to the above choice of $\eta_1$ and $\eta_2$, we see that
\begin{align}
	\|\varphi(t)-\varphi _\ast\|_{H^2}\leq \omega,\quad\forall\, t\in[0,T_2].\label{b20}
\end{align}
Using \eqref{p1}, \eqref{SoB}, \eqref{b2}, \eqref{omega}, \eqref{eta2} and \eqref{b20}, we further get
\begin{align}
	&\|\varphi(t)\|_{C(\overline{\Omega})}
    \leq \|\varphi _\ast\|_{C(\overline{\Omega})}+C_\mathrm{S}\|\varphi(t)-\varphi _\ast\|_{H^2}
    \leq 1-\frac{2\xi}{3},\quad\forall\,t\in[0,T_2],\label{b21}\\
	&\|\varphi(t)-\varphi _0\|_{H^2}
    \leq \|\varphi(t)-\varphi _\ast\|_{H^2}+\|\varphi _0-\varphi _\ast\|_{H^2}
    \leq \frac{ E_0}{M_2 C_\ast},\quad\forall\,t\in[0,T_2].\label{b22}
\end{align}
Finally, from \eqref{b10}, \eqref{n2}, \eqref{b15}, \eqref{b16}, \eqref{omega} and \eqref{b22}, we can obtain
\begin{align}
	\int_0^{T_2}\widetilde{\mathcal{H}}(t)\,\mathrm{d}t
	&\leq \int_0^{T_2}\Big(\nu^\ast\|\mathbb{D}\mathbf{v}(t)\|_{L^2}^2+\frac{\beta_\ast b^\ast}{2}\|\nabla\mu(t)\|_{L^2}^2
    +\beta_\ast\|\varphi(t)\|_{L^\infty} \|\nabla\mu(t)\|_{L^2}\|\mathbf{v}(t)\|_{L^2}\Big)\,\mathrm{d}t\notag\\
	&\leq \int_0^{T_2}\Big(\nu^\ast\|\mathbb{D}\mathbf{v}(t)\|_{L^2}^2+\frac{\beta_\ast b^\ast}{2}\|\nabla\mu(t)\|_{L^2}^2
    +C_\mathrm{P}\beta_\ast\|\nabla\mu(t)\|_{L^2} \|\nabla\mathbf{v}(t)\|_{L^2}\Big)\,\mathrm{d}t\notag\\
	&\leq \int_0^{T_2}\Big(\nu^\ast\|\mathbb{D}\mathbf{v}(t)\|_{L^2}^2+\frac{\beta_\ast b^\ast}{2}\|\nabla\mu(t)\|_{L^2}^2
    +\frac{\beta_\ast C_\mathrm{P}^2}{b^\ast}\|\nabla\mathbf{v}(t)\|_{L^2}^2+\frac{\beta_\ast b^\ast}{4}\|\nabla\mu(t)\|^2_{L^2}\Big)\,\mathrm{d}t\notag\\
	&\leq\int_0^{T_2}\Big(2\nu^\ast\|\mathbb{D}\mathbf{v}(t)\|_{L^2}^2+b^\ast\|\nabla\mu(t)\|^2_{L^2}\Big)\,\mathrm{d}t\notag\\
	&\leq \Big(\frac{\nu^\ast}{\nu_\ast}+\frac{b^\ast}{b_\ast}\Big)\Big(\mathcal{E}(0)-\mathcal{E}(T_2)\Big)\notag\\
	&=C_\ast\Big(\frac{1}{2}\Big\|\sqrt{\rho(\varphi_0)}\mathbf{v}_0\Big\|_{L^2}^2-\frac{1}{2}\Big\|\sqrt{\rho(\varphi(T_2))}\mathbf{v}(T_2)\Big\|_{L^2}^2
    +E_{\text{free}}(\varphi _0)-E_{\text{free}}(\varphi(T_2))\Big)\notag\\
	&\leq C_\ast \Big(\frac{\rho^\ast\eta_1^2}{2}+|E_{\text{free}}(\varphi _0)-E_{\text{free}}(\varphi(T_2))|\Big)\notag\\
	&\leq C_\ast\Big(\frac{\omega^2}{32}+M_2\|\varphi _0-\varphi(T_2)\|_{H^2}\Big)\notag\\
	&\leq C_\ast\Big(\frac{\omega^2}{32}+M_2\frac{C_\ast^{-1} E_0}{M_2}\Big)\notag\\
	&\leq C_\ast\Big(E_0 C_\ast^{-1}+M_2\frac{C_\ast^{-1} E_0}{M_2}\Big)\leq 2E_0.\label{b23}
\end{align}

\subsection{Global strong solutions and Lyapunov stability}
We are in a position to prove Theorem \ref{global-strong}.
\medskip

\emph{Step 1.}
According to Section \ref{refine-es}, we have shown that problem \eqref{AGG}--\eqref{BC-IC} admits a unique strong solution on $[0,T_2]$ that satisfies the estimates \eqref{b20}--\eqref{b23}.
From \eqref{b12} and \eqref{b23}, we infer that
\begin{align*}
	\int_{\frac{1}{2}T_2}^{T_2}\widetilde{\mathcal{H}}(t)\,\mathrm{d}t\leq \int_{0}^{T_2}\widetilde{\mathcal{H}}(t)\,\mathrm{d}t\leq 2E_0.
\end{align*}
Thanks to the definition of $E_0$, there exists a time
$t^\ast\in[T_2 /2,T_2]$,
such that
\begin{align}
	\widetilde{\mathcal{H}}(t^\ast)\leq \min\{1,M_1\}.\label{b25}
\end{align}
From \eqref{0205-mu}, \eqref{b12}, \eqref{b13} and \eqref{b25}, we obtain the following estimates
\begin{equation}
\begin{aligned}
&\|\mu(t^\ast)\|_{H^1}\leq \frac{2L_3}{\sqrt{\beta_\ast b_\ast}}\sqrt{\widetilde{\mathcal{H}}(t^\ast)}+L_3\leq  \frac{2L_3}{\sqrt{\beta_\ast b_\ast}}+L_3,\\
&\|\mathbf{v}(t^\ast)\|_{H^1}\leq \frac{8C_\mathrm{P}}{\sqrt{\nu_\ast}}\sqrt{\widetilde{\mathcal{H}}(t^\ast)}\leq\frac{8C_\mathrm{P}}{\sqrt{\nu_\ast}}.
\end{aligned}
\label{b27}
\end{equation}
Recalling the definition of $M$ (cf. \eqref{M}), we find
\begin{align}
	\|\mu(t^\ast)\|_{H^1}\leq M,\quad\|\mathbf{v}(t^\ast)\|_{H^1}\leq M.\label{b28}
\end{align}
Hence, $(\mathbf{v}(t^\ast),\mu(t^\ast))$ satisfies the same bound as the initial data $(\mathbf{v}_0,\mu_0)$ (comparing \eqref{b1} with \eqref{b28}). In addition, we have
$\mathcal{E}(t^\ast)\leq \mathcal{E}(0)$,
and the same strict separation property \eqref{b21}.
\medskip

\emph{Step 2. } Now we take $t^\ast$ as the initial time. Similarly to Step 1, we can obtain a local strong solution on the time interval $[t^\ast,t^\ast+T_2]$.
By the uniqueness of local strong solutions, we indeed obtain a (unique) local strong solution on $[0,t^\ast+T_2]$,
and this solution satisfies the estimates \eqref{0205-mu}, \eqref{b13} on $[0,t^\ast+T_2]$.
Then, we can derive refined estimates for the local strong solution on $[0,\frac{3}{2}T_2]$ as \eqref{b15}--\eqref{b23}.
Under the same choice of $\eta_1$ and $\eta_2$, we are able to derive an estimate similar to \eqref{b23}, that is,
$$
\int_0^{\frac{3}{2}T_2}\widetilde{\mathcal{H}}(t)\,\mathrm{d}t\leq 2E_0.
$$
Observing that $\widetilde{\mathcal{H}}(t)\geq0$, like in Step 1, we can find a new time
$t^{\ast\ast}\in [T_2,\frac{3}{2}T_2]$
satisfying \eqref{b25} and \eqref{b27}.
In addition, we can derive a refined higher-order estimate \eqref{b28} at this new time $t^{\ast\ast}$.
By iteration, we are able to construct a global strong solution with expected properties as in the statement of Theorem \ref{global-strong}. The proof is complete.
\hfill $\square$

%%%%%%%%%%%%%%%%%%%%%%%%%%%%%%%%%%%%%%%%%%%%%%%%%

\appendix
\renewcommand{\appendixname}{}
\renewcommand{\thesection}{\Alph{section}}

\section{Useful Tools}
\label{Appendix}
\setcounter{equation}{0}

In the following, we report some useful tools that have been used in this study.

\subsection{\texorpdfstring{Global mollifier in $\Omega$}{Global mollifier in \Omega}}
\label{APP-modif}
We first recall the following local approximation:
if $f\in L^p(0,T;W^{1,p}(\Omega))$, then
\begin{align}
    f^{\varepsilon}\to f\quad\text{in }\ L^p_{\mathrm{loc}}(0,T;W_\mathrm{loc}^{1,p}(\Omega))
    \quad \forall\, p\in[1,\infty),\quad\text{as }\varepsilon\to0,\notag%\label{convergence-1}
\end{align}
where
\begin{equation}
\begin{aligned}
    & f^\varepsilon(x,t)=\int_0^T\int_\Omega f(y,s)\eta_{\varepsilon}(x-y,t-s)\,\mathrm{d}y\,\mathrm{d}s, \\
    & \eta_{\varepsilon}(x,t)=\frac{1}{\varepsilon^4}\eta\Big(\frac{x}{\varepsilon},\frac{t}{\varepsilon}\Big),
\end{aligned}
\label{time-mollifier}
\end{equation}
with $\eta(x,t)$ being the standard mollifier supported in the  unit ball around $({\bf 0}, 0)$ in $\mathbb{R}^3\times \mathbb{R}$.
Here, the subscript ``loc'' means that the convergence holds in
$L^p(I;W^{1,p}(U))$ for any open sets $I\subset\subset (0,T)$ and $U\subset\subset \Omega$.

To deal with boundary effects,
we apply a result in \cite{CLWX2020} on the construction of a global mollifier in $\Omega$.
Since $\partial\Omega\in C^1$, for a fixed point $\mathrm{P}\in \partial\Omega$,
there exist some $r_\mathrm{P}>0$ and a $C^1$ function $h:\mathbb{R}^2\to\mathbb{R}$ such that,
upon relabeling the coordinate axis if necessary, we have
\[\Omega\cap B(\mathrm{P},r_\mathrm{P})=\{x\in B(\mathrm{P},r_\mathrm{P}): x_3>h(x_1,x_2)\},\]
where $B(\mathrm{P},r_\mathrm{P})$ is an open ball centered at $\mathrm{P}$ with radius $r_\mathrm{P}$.
Since the boundary $\partial\Omega$ is compact, there are finitely many points $\mathrm{P}_i\in\partial\Omega$, radii $r_i>0$,
and corresponding sets $V_i=\Omega\cap B(\mathrm{P}_i, r_i/2)$ for $i=1,...,k$, such that $\partial\Omega\subset \cup_{i=1}^k \overline{V_i}$.
We also take an open set $V_0\subset\subset\Omega$ that satisfies $\Omega\subset \bigcup_{i=0}^k V_i$.
For $\lambda>0$ and $0<\varepsilon< r_{\ast}/8$ where $r_\ast:=\min\{r_1,...,r_k\}$, we define the shifted point
\begin{align}
	x^\varepsilon:= x-\lambda\varepsilon \mathbf{n}(\mathrm{P}_i),\notag
\end{align}
for all
\begin{align}
x\in \widetilde{V}_i:= V_i+B(0,\varepsilon):=\{x=x'+x'':x'\in V_i,\ x''\in B(0,\varepsilon)\},\notag
\end{align}
where $\mathbf{n}(\mathrm{P}_i)$ is the unit outward normal vector of $\partial\Omega$ at $\mathrm{P}_i$. Observe that for some fixed, sufficiently large number $\lambda>0$,
\begin{align}
    B(x^\varepsilon,\varepsilon)\subset \Omega\cap B(\mathrm{P}_i,r_i)\label{containing}
\end{align}
for all $x\in V_i$ $(i=1,...,k)$ and all small $\varepsilon>0$ (see \cite[Section 5.3]{Evans}).
Define the shifted function
\[\widetilde{f}(x,t)=f(x^\varepsilon,t),\qquad\forall\, (x,t)\in \widetilde{V}_i\times (0,T).\]
In view of \eqref{containing}, we see that, for all $x\in \widetilde{V}_i$, it holds
\[x^\varepsilon=x-\lambda\varepsilon \mathbf{n}(\mathrm{P}_i)=x'-\lambda\varepsilon \mathbf{n}(\mathrm{P}_i)+x''\in B((x')^\varepsilon,\varepsilon)\subset \Omega\cap B(\mathrm{P}_i,r_i),\]
where $(x')^\varepsilon=x'-\lambda\varepsilon\mathbf{n}(\mathrm{P}_i)$.
Consequently, $\widetilde{f}(x,t)$ is well-defined in $\widetilde{V}_i\times(0,T)$.
There is room to mollify $\widetilde{f}(x,t)$ like \eqref{time-mollifier}, that is,
\begin{align}
	\widetilde{f}_i^{\varepsilon}(x,t)
    & =\int_0^T \int_{B(x,\varepsilon)}\widetilde{f}(y,s)\eta_{\varepsilon}(x-y,t-s)\,\mathrm{d}y\,\mathrm{d}s
    \notag \\
    & =\int_0^T \int_{ B(x^\varepsilon,\varepsilon)} f(y,s)\eta_{\varepsilon}(x^\varepsilon-y,t-s)\,\mathrm{d}y\,\mathrm{d}s,\label{mollifies-shifted}
\end{align}
for every $(x,t)\in V_i\times(\varepsilon,T-\varepsilon)$.
Then we have (cf. \cite[(2.6)]{CLWX2020}):
\begin{align}
\lim_{\varepsilon\to0}\| \widetilde{f}_i^{\varepsilon}-f\|_{L^p_{\mathrm{loc}}(0,T;W^{1,p}(V_i))}=0 \quad\text{for }i=1,...,k.\notag
\end{align}

Now, we report the following result (see, e.g., \cite[Proposition 2.1]{CLWX2020}).

\begin{pro}
    Let $\{\xi_i\}_{i=0}^k$ be a smooth partition of unity subordinate to $V_i$, that is,
    \[0\leq \xi_i\leq1,\ \ \xi_i\in C^\infty(V_i),\ \ \mathrm{supp}
\,\xi_i\subset V_i,\ \ \sum_{i=0}^k \xi_i=1\ \ \text{in }\Omega.\]
Define
\begin{align}
    [f]^{\varepsilon}(x,t)=\xi_0(x)f^{\varepsilon}(x,t)+\sum_{i=1}^k\xi_i(x)\widetilde{f}_i^{\varepsilon}(x,t), \quad\forall\,(x,t)\in\Omega\times(0,T),\label{mollify-profile}
\end{align}
where $f^{\varepsilon}$ and $\widetilde{f}_i^{\varepsilon}$ are defined in \eqref{time-mollifier} and \eqref{mollifies-shifted}, respectively.
Then, sending $\varepsilon\to0$, we have
\begin{align}
    [f]^{\varepsilon}\to f\quad\text{in }L^p_\mathrm{loc}(0,T;W^{1,p}(\Omega)).\label{mollify-convergence}
\end{align}
Moreover, as $\varepsilon\to0$, it holds
\begin{equation}
\begin{aligned}
   & [f]^{\varepsilon}-f^{\varepsilon}\to0 \quad\text{in }L^p_\mathrm{loc}(0,T;W^{1,p}(V_0)),
   \\
   & [f]^{\varepsilon}-\widetilde{f}_i^{\varepsilon}\to0 \quad\text{in }L^p_\mathrm{loc}(0,T;W^{1,p}(V_i)).
\end{aligned}
   \label{local-convergence-1}
\end{equation}
\end{pro}

\subsection{Commutator estimates and Hardy-type embedding}

First, we recall two lemmas on the commutator estimates, see \cite[Lemma 2.3]{Lions} and \cite[Lemma 2.2]{CLWX2020}.

\begin{lem}
Suppose $\rho\in W^{1,r_1}(0,T;W^{1,r_1}(\Omega))$, $u\in L^{r_2}([0,T]\times\Omega)$,
and $1\leq r, r_1, r_2\leq \infty$, $1/ r_1+1/r_2=1/r$. Then we have
\begin{align}
	&\|\partial((\rho u)^{\varepsilon}-\rho u^{\varepsilon})\|_{L_{\mathrm{loc}}^r([0,T]\times\Omega)}
    \leq C\|u\|_{L^{r_2}([0,T]\times\Omega)}\|\partial\rho\|_{L^{r_1}([0,T]\times\Omega)},\notag\\
    &\|\partial((\widetilde{\rho}\widetilde{u})_i^{\varepsilon}-\rho\widetilde{u}_i^{\varepsilon})\|_{L_{\mathrm{loc}}^r(0,T;L^r(V_i))}
    \leq C\|u\|_{L^{r_2}([0,T]\times\Omega)}\|\partial\rho\|_{L^{r_1}([0,T]\times\Omega)},\notag%\label{commutator-2}
\end{align}
where $\partial=\partial_t$ or $\partial=\nabla$, and $V_i$ $(i=1,...,k)$ are the same as mentioned above.
Furthermore,
\begin{align}
	&\partial((\rho u)^{\varepsilon}-\rho u^{\varepsilon})\to0
    \quad\text{in} \  L_{\mathrm{loc}}^{\underline{r}}([0,T]\times\Omega)\quad\text{as }\varepsilon\to0,\label{commutator-convergence-1}\\
    &\partial((\widetilde{\rho}\widetilde{u})_i^{\varepsilon}-\rho\widetilde{u}_i^{\varepsilon})\to0
    \quad \text{in} \ L^{\underline{r}}_{\mathrm{loc}}(0,T;L^{\underline{r}}(V_i))\quad\text{as }\varepsilon\to0,\label{commutator-convergence-2}	
\end{align}
where $\underline{r}=r$ if $r_2<\infty$ and $\underline{r}<r$ if $r_2=\infty$.
\end{lem}

\begin{lem}
	\label{Commutator-lemma-3}
	Let $1\leq r, r_1, r_2< \infty$, $1/ r_1+1/r_2=1/r$
    and $f\in L^{r_1}((0,T)\times\Omega)$, $g\in L^{r_2}((0,T)\times\Omega)$.
    Then we have
	\begin{align}
		(fg)^{\varepsilon}- fg^{\varepsilon}\to0\quad\text{in}\ L^{r}_{\mathrm{loc}}((0,T)\times\Omega)
        \quad\text{as }\varepsilon\to0.\notag%\label{commutator-convergence-3}
	\end{align}
	
\end{lem}

The next lemma gives a Hardy-type embedding inequality (see \cite[p.~313]{Brezis}).
\begin{lem}
	\label{Hardy}
	Let $p\in(1,\infty)$ and $f\in W_{0}^{1,p}(\Omega)$.
    There is a constant $C$, depending on $p$ and $\Omega$, such that
	\begin{align*}
		\Big\|\frac{f(x)}{\mathrm{dist}(x,\partial\Omega)}\Big\|_{L^p}\leq C\|f\|_{W_0^{1,p}}.
	\end{align*}
\end{lem}

\subsection{{\L}ojasiewicz--Simon inequality}

 The following gradient inequality for the Cahn--Hilliard equation with a non-constant gradient energy coefficient is crucial in the proof of Theorem \ref{global-strong}  (see \cite[Theorem 1.1]{CGGS}).

\begin{lem}
	\label{LS}
	Let $\Omega\subset\mathbb{R}^3$ be a bounded smooth domain.
    Assume that $a(\cdot)$ is real analytic on the open interval $(-1,1)$.
    Then, for any  $\psi \in\mathcal{S}_m$, there exist constants $\theta\in(0,1/2)$, $C_L>0$ and $\kappa>0$
    such that
	the following inequality holds
	\begin{align}
    & |E_{\mathrm{free}}(\varphi)-E_{\mathrm{free}}(\psi)|^{1-\theta}
    \notag \\
    &\quad \leq C_L \Big\|-\mathrm{div}(a(\varphi)\nabla\varphi)+\frac{a'(\varphi)}{2}|\nabla\varphi|^2+\Psi'(\varphi)
    -\overline{\frac{a'(\varphi)}{2}|\nabla\varphi|^2+\Psi'(\varphi)}\Big\|_{L^2},
    \notag
    \end{align}
    for all $\varphi\in H^2(\Omega)$ such that $\overline{\varphi}=m$, $\partial_\mathbf{n}\varphi=0$ on $\partial\Omega$ and $\|\varphi-\psi\|_{H^2}\leq \kappa$.
\end{lem}

\section*{Declarations}
\noindent \textbf{Acknowledgments.}
The authors would like to thank the anonymous referees for their careful reading of an initial version of this paper and for many helpful comments that
allowed them to improve the presentation. This work has benefited from some stimulating discussions between H. Garcke and H. Wu during the Thematic Program on Free Boundary Problems at the Erwin Schr\"{o}dinger International Institute for Mathematics and Physics (ESI) in Vienna, whose hospitality is kindly acknowledged. H. Wu is a member of the Key Laboratory of Mathematics for Nonlinear Sciences (Fudan University), Ministry of Education of China.
\smallskip

\noindent \textbf{Funding.}
H. Wu was partially supported by the Natural Science Foundation of Shanghai under grant number 25ZR1401023.
\smallskip

\noindent \textbf{Competing interests.}
The authors have no relevant financial or non-financial interests to disclose.
\smallskip

\noindent \textbf{Data availability statement.}
Data sharing is not applicable to this article, as no datasets were generated or analyzed during the current study.
\smallskip

\noindent \textbf{Declaration of Generative AI use.} The authors have used GPT-6.1 Sol for copyediting.

\end{document}